\documentclass[11pt,letterpaper]{amsart}
\usepackage{amsmath,amssymb}
\usepackage{bbm}
\usepackage{graphicx,tikz,mathtools}
\usepackage{soul}
\usepackage[pdftex]{hyperref}
\usepackage{cite}
\hypersetup{
	colorlinks=true,
	linkcolor=blue, 
	citecolor=blue,
	filecolor=blue,
	urlcolor=blue,
}
\newtheorem{theorem}{Theorem}[section]
\newtheorem{proposition}[theorem]{Proposition}
\newtheorem{corollary}[theorem]{Corollary}

\newtheorem{lemma}[theorem]{Lemma}
\newtheorem*{conjecture*}{Conjecture}

\theoremstyle{definition}

\theoremstyle{remark}
\newtheorem{remark}[theorem]{Remark}

\newcommand{\al}{\alpha}
\newcommand{\be}{\beta}
\newcommand{\de}{\delta}
\newcommand{\ep}{\varepsilon}

\newcommand{\ga}{\gamma}

\newcommand{\la}{\lambda}

\newcommand{\si}{\sigma}
\newcommand{\te}{\theta}
\newcommand{\vp}{\varphi}

\newcommand{\De}{\Delta}
\newcommand{\Ga}{\Gamma}
\newcommand{\La}{\Lambda}

\newcommand{\Om}{\Omega}

\def\CC{\mathbb{C}}
\def\NN{\mathbb{N}}
\def\RR{\mathbb{R}}

\def\ZZ{\mathbb{Z}}

\def\TT{\mathbb{T}}

\renewcommand\SS{\mathbb{S}}

\newcommand{\dds}{\frac{{\rm d}}{{\rm d}s}}

\newcommand{\cB}{{\mathcal B}}
\newcommand{\cC}{{\mathcal C}}

\newcommand{\cH}{{\mathcal H}}
\newcommand{\cI}{{\mathcal I}}
\newcommand{\cJ}{{\mathcal J}}

\newcommand{\cR}{{\mathcal R}}

\newcommand{\cX}{{\mathcal X}}
\newcommand{\cY}{{\mathcal Y}}

\newcommand\nuT{\widetilde{\nu}}
\newcommand\psiT{\widetilde{\Psi}}
\newcommand\LaT{\widetilde{\Lambda}}
\newcommand\phiT{\widetilde{\Phi}}
\newcommand{\pd}{\partial}
\newcommand\minus\backslash

\newcommand\lan\langle
\newcommand\ran\rangle

\newcommand{\sign}{\operatorname{sign}}

\newcommand{\supp}{\operatorname{supp}}

\newcommand{\Ker}{\operatorname{Ker}}

\DeclareMathOperator\Div{div}

\DeclareMathOperator\dist{dist}

\newcommand\MU{\mu_{0,2}}
\newcommand\LA{\la_{l,0}}

\newcommand{\Xl}{\mathcal{X}}
\newcommand{\Yl}{\mathcal{Y}}

\renewcommand\leq\leqslant
\renewcommand\geq\geqslant
\newcommand\ev{_{l}}
\newcommand\D{_{\mathrm{D}}}
\newcommand\DN{_{\mathrm{DN}}}

\newcommand\BOm{\overline\Om}

\numberwithin{equation}{section}

\newcommand\hrho{\widehat\rho}
\newcommand\hsi{\widehat\si}
\newcommand{\txi}{\widetilde{\xi}}
\newcommand{\teta}{\widetilde{\eta}}
\newcommand{\tk}{\tilde{k}}
\newcommand{\tde}{\tilde{\delta}}
\newcommand{\tsi}{\widetilde{\si}}
\newcommand{\tcJ}{\widetilde{\cJ}}

\begin{document}

\title[A Schiffer-type problem for annuli]{A Schiffer-type problem for annuli \\ with applications to stationary planar Euler flows}

\author[A. Enciso]{Alberto Enciso}
\address{  
\newline
\textbf{{\small Alberto Enciso}} 
\vspace{0.15cm}
\newline \indent Instituto de Ciencias Matem\'aticas, Consejo Superior de Investigaciones \newline \indent Cient\'\i ficas, 28049 Madrid, Spain}
\email{aenciso@icmat.es}

 \author[A. J. Fern\'andez]{Antonio J.\ Fern\'andez}
 \address{ \vspace{-0.4cm}
\newline 
\textbf{{\small Antonio J. Fern\'andez}} 
\vspace{0.15cm}
\newline \indent Departamento de Matem\'aticas, Universidad Aut\'onoma de Madrid, 28049 \newline \indent Madrid, Spain}
 \email{antonioj.fernandez@uam.es}

 \author[D. Ruiz]{David Ruiz}
 \address{ \vspace{-0.4cm}
\newline 
\textbf{{\small David Ruiz}} 
\vspace{0.15cm}
\newline \indent  
 IMAG, Departamento de An\'alisis Matem\'atico, Universidad de Granada, \newline \indent 18071 Granada, Spain}
 \email{daruiz@ugr.es}

 \author[P. Sicbaldi]{Pieralberto Sicbaldi}
 \address{ \vspace{-0.4cm}
 \newline 
\textbf{{\small Pieralberto Sicbaldi}} 
\vspace{0.15cm}
\newline \indent IMAG, Departamento de An\'alisis Matem\'atico, Universidad de Granada, \newline \indent 18071 Granada, Spain 
\newline \indent Aix-Marseille Universit\'e - CNRS, Centrale Marseille - I2M, Marseille, \newline \indent France}
 \email{pieralberto@ugr.es}



%%    General info
%\subjclass[2010]{35B38, 58J05, 58K45}
%\date{\today}
%
%\keywords{ }
%
\begin{abstract}
\emph{If on a smooth bounded domain~$\Om\subset\RR^2$ there is a (nonconstant) Neumann Laplace eigenfunction $u$ that is locally constant on the boundary, must $\Om$ be a disk or an annulus?}\/ This question can be understood as a weaker analog of the well known Schiffer conjecture, in that the eigenfunction $u$ is allowed to take a different constant value on each connected component of $\pd \Om$ yet many of the known rigidity properties of the original problem are essentially preserved. Our main result provides a negative answer by constructing a family of nontrivial doubly connected domains~$\Om$ with the above property. Furthermore, our construction implies the existence of continuous, compactly supported stationary weak solutions to the 2D incompressible Euler equations which are not locally radial.  
\end{abstract}

\maketitle

\vspace{-0.5cm}

\section{Introduction}

One of the most intriguing problems in spectral geometry is the
so-called Schiffer conjecture, that dates back to the 1950s. In his 1982
list of open problems, S.T. Yau stated it (in the case $n = 2$) as
follows~\hspace{-0.01cm}\cite[Problem 80]{Yau}:

\begin{conjecture*}
If a nonconstant Neumann eigenfunction~$u$ of the Laplacian on a smooth bounded domain~$\Om\subset\RR^n$ is constant on the boundary~$\pd\Om$, then $u$ is radially symmetric and~$\Om$ is a ball.
\end{conjecture*}

This overdetermined problem is closely related to the {\it Pompeiu problem}~\cite{Pompeiu}, an open question in integral geometry with many applications in remote sensing, image recovery and tomography~\cite{Berenstein, WillmsGladwell, SSW}. The Pompeiu problem can be stated as the following inverse problem: Given a bounded domain $\Om\subset\RR^n$, is it possible to recover any continuous function~$f$ on~$\RR^n$ from knowledge of its integral over all the domains that are the image of~$\Om$ under a rigid motion? If this is the case, so that the only $f\in C(\RR^n)$ satisfying
\begin{equation}\label{E.Pompeiu}
\int_{\cR(\Om)}f(x)\, dx=0\,,
\end{equation}
for any rigid motion~$\cR$ is $f\equiv 0$,
 the domain~$\Om$ is said to have the {\em Pompeiu property}\/. Squares, polygons, convex domains with a corner, and ellipses have the Pompeiu property, and Chakalov was apparently the first to point out that balls fail to have the Pompeiu property~\cite{chakalov, L.BrownSchreiberTaylor, Zalcman}. In 1976, Williams proved \cite{Williams76} that a smooth bounded domain with boundary homeomorphic to a sphere fails to have the Pompeiu property if and only if it supports a nontrivial Neumann eigenfunction which is constant on $\partial \Omega$. Therefore, the Schiffer conjecture and the Pompeiu problem are equivalent for this class of domains\footnote{For domains with different topology the connection between the Schiffer and Pompeiu problems is less direct, and in fact one can construct many domains without the Pompeiu property using balls of different centers and radii.}.

Although the Schiffer conjecture is famously open, some partial results are available. It is known that $\Om$~must indeed be a ball under one of the following additional hypotheses:
\begin{enumerate}
\item There exists an infinite sequence of orthogonal Neumann eigenfunctions that are constant on $ \partial \Om$, which is connected \cite{Berenstein, BerensteinYang}.
\item The third order interior normal derivative of $u$ is constant on $\pd\Om$, which is connected \cite{Liu}.
\end{enumerate}
In dimension 2, some further results are available, and the conjecture has been shown to be true in some other special cases:
\begin{enumerate}
\setcounter{enumi}{2}
\item When $\Omega$ is simply connected and $u$ has no saddle points in the interior of $\Om$ \cite{WillmsGladwell}.
\item If $\Om$ is simply connected and the eigenvalue~$\mu$ is among the seven lowest Neumann eigenvalues of the domain \cite{aviles, Deng}.
\item If the fourth or fifth order interior normal derivative of $u$ is constant on $\pd\Om$ \cite{KawohlLucia}.
\end{enumerate}
It is also known that the boundary of any reasonably smooth domain $\Om\subset\RR^n$ with the property stated in the Schiffer conjecture must be analytic as a consequence of a result of Kinderlehrer and Nirenberg~\cite{KinderlehrerNirenberg} on the regularity of free boundaries. 

%Moreover, in the statement of the conjecture it is enough to suppose that $\pd\Om$ is Lipschitz continuous, because in this case the existence of the desired Neumann eigenfunction implies that $\pd\Om$ is smooth and then real-analytic, see \cite{Caffarelli, Williams81}.

If one considers domains with disconnected boundary, it is natural to wonder what happens if one relaxes the hypotheses in the Schiffer conjecture by allowing the eigenfunction to be {\em locally constant}\/ on the boundary, that is, constant on each connected component of~$\pd\Om$. Of course, the radial Neumann eigenfunctions of a ball or an annulus are locally constant on the boundary, so the natural question in this case is whether $\Om$ must be necessarily a ball or an annulus.

We emphasize that this question shares many features with the Schiffer conjecture. First, essentially the same arguments show that $\pd\Om$ must be analytic. Moreover, in the spirit of~\cite{Berenstein, BerensteinYang}, we show that if there exists an infinite sequence of orthogonal eigenfunctions that are locally constant on the boundary of $\Om \subset \RR^2$, then $\Omega$ must be a disk or an annulus. Furthermore, we prove that if the Neumann eigenvalue is sufficiently low, $\Omega$~must be a disk or an annulus, which is a result along the lines of~\cite{aviles, Deng}. These results are interesting in their own right, as they are nontrivial adaptations of the arguments used for the Schiffer conjecture. Precise statements and proofs of these rigidity results can be found in Section~\ref{S.Rigidity}.

Our main objective in this paper is to give a negative answer to this question. Indeed, we can prove the following: 
\begin{theorem}\label{T.main}
There exist parametric families of doubly connected bounded domains $\Om\subset\RR^2$ such that the Neumann eigenvalue problem
$$
\Delta u + \mu u = 0\quad in\ \Omega\,, \qquad \pd_\nu u = 0 \quad on \ \partial\Omega\,,
$$
%$$ \left \{ \begin{array}{ll} \Delta u + \mu u =0, & x \in \Omega, \\ \nabla u=0, & x \in \partial \Om,\end{array} \right.$$
admits, for some $\mu \in \RR$, a non-radial eigenfunction that is locally constant on the boundary (i.e., $\nabla u \equiv 0$ on~$\pd\Om$).
More precisely,  for any large enough integer $l $ and for all $s$ in a small neighborhood of $0$, the family of domains $\Omega\equiv \Om_{l,s}$ is given in polar coordinates by 
	\begin{equation}\label{E.Omthm}
	\Om:=\big\{(r,\te)\in\RR^+\times \TT: a_l+s\, b_{l,s}(\te)<r<1+s\, B_{l,s}(\te)\big\}\,,
	\end{equation}
where $b_{l,s},B_{l,s}$ are analytic functions on the circle $\TT:=\RR/2\pi\ZZ$ of the form
\[
b_{l,s}(\te)= \alpha_l\cos l\te + o(1)\,,\qquad B_{l,s}(\te)= \be_l \cos l\te + o(1)\,,
\]
where $a_l\in(0,1)$, $\al_l$ and $\beta_l$ are nonzero constants, and where the $o(1)$ terms tend to~$0$ as $s\to0$.
\end{theorem}

 We refer to Theorem~\ref{T.mainDetails} in the main text for a more precise statement.

%\subsection*{Connection with a Pompeiu-type problem} As mentioned before, the Schiffer conjecture and the Pompeiu problem are closely related. Specifically, Williams proved~\cite{Williams76} that a smooth simply-connected domain~$\Om$ fails to have the Pompeiu property if and only if there is a Neumann eigenfunction of~$\Om$ that is constant on~$\pd\Om$. When the domain is not simply connected, the connection is less direct (see e.g.~\cite{Williams76} and references therein). For instance, if $R_1>R_2$ are zeros of the Bessel function $J_{n/2}$, it is easy to show that a family of examples of non-spherically symmetric domains~$\Om$ without the Pompeiu property is the disk $B(0,R_1)$ centered at the origin and of radius~$R_1$ with (one or several) disjoint closed balls of radius $R_2$ deleted, provided that the latter are contained in~$B(0,R_1)$. A more complicated example was provided by Federbush. However, we are not aware of any domains without the Pompeiu property that are not constructed using balls of radii given by zeros of~$J_{n/2}$. 

 In view of the connection between the Schiffer conjecture and the Pompeiu problem, it stands to reason that the domains of the main theorem should be connected with an integral identity somewhat reminiscent of the Pompeiu property. Specifically, one can prove the following corollary of our main result, which is an analog of the identity~\eqref{E.Pompeiu} in which one replaces the indicator function~$\mathbbm 1_\Om$ implicit in the integral by a linear combination of indicator functions:

\begin{corollary}\label{C.Pompeiu}
Let ~$\Om$ be a domain as in Theorem~\ref{T.main} and let $\Om':=\{r<a_l+ s\, b_{l,s}(\te)\}$ denote the bounded component of $\RR^2\backslash\BOm$. Then there exists a positive constant $c$ and a nonzero function $f \in C^{\infty}(\RR^2)$ such that
\[
\int_{\cR(\Om)}f\, dx - c\int_{\cR(\Om')}f\, dx=0\,,
\]
for any rigid motion~$\cR$.	
\end{corollary}

\subsection*{Non-radial compactly supported solutions to the 2D Euler equations}
Theorem \ref{T.main} implies the existence of a special type of stationary 2D Euler flows, which was the original motivation of our work. Let us now describe this connection in detail.

It is well known that if a scalar function $w$ satisfies a semilinear equation of the form 
\begin{equation}\label{E.Deg}
\De w+ g(w)=0 \quad \textup{in } \RR^2\,,
\end{equation}
for some function~$g$, then the velocity field and pressure function given by
\[
v:=\left(\frac{\pd w}{\pd x_2},-\frac{\pd w}{\pd x_1}\right)\,,\qquad p:= -\frac12|\nabla w|^2- \int_0^w g(s) \, ds\,,
\]
define a stationary solution to the Euler equations:
\begin{equation}
\label{euler} v\cdot \nabla v+\nabla p=0\quad\textup{ and } \quad \Div v=0 \quad \textup{ in } \RR^2\,.
\end{equation}

Any function $w$ that is radial and compactly supported gives rise to a compactly supported stationary Euler flow in two dimensions, even if it does not solve a semilinear elliptic equation as above. The stationary solutions that are obtained by patching together radial solutions with disjoint supports are referred to as \textit{locally radial}.  For a detailed bibliography on the rigidity and flexibility properties of stationary Euler flows in two dimensions and their applications, we refer to the recent papers~\cite{JGS, gomez}.

Continuous compactly supported stationary Euler flows in two dimensions that are not locally radial were constructed just very recently~\cite{JGS}. They are of vortex patch type and the velocity field is piecewise smooth but not differentiable\footnote{Therefore, the Euler equations are to be understood in the weak sense. Details are given in ~Section~\ref{S.corollaries}.}. The proof of this result is hard: it involves several clever observations and a Nash--Moser iteration scheme. Furthermore, as emphasized in~\cite{JGS}, these solutions cannot be obtained using an equation of the form~\eqref{E.Deg}. In the setting of rough solutions~\cite{Chof}, a wealth of compactly supported stationary weak solutions to the Euler equations in dimensions~2 and higher can be constructed using convex integration methods, but in this case the velocity field is only in~$L^\infty$. Also, a straightforward consequence of the recent paper~\cite{Ruiz-arxiv} is the existence of discontinuous but piecewise smooth stationary Euler flows that are not locally radial. On the other hand, the existence of non-radial stationary planar Euler flows is constrained by several recent rigidity results~\cite{Hamel, gomez, R}. The existence of $C^1$~compactly supported solutions of~\eqref{euler} which are not locally radial is not yet known; if the support condition is relaxed to merely having finite energy, the existence follows directly from~\cite{Musso}.

One can readily use Theorem~\ref{T.main} to construct families of continuous stationary planar Euler flows with compact support that are not locally radial. These solutions, which are not of patch type, are analytic in the interior of their support but fail to be differentiable across its boundary.

\begin{theorem}\label{T.Euler}
	Let $\Om$ be a domain as in Theorem~\ref{T.main}. Then there exists a continuous, compactly supported stationary solution $(v,p)$ to the incompressible Euler equations on the plane such that $\supp v=\BOm$. Furthermore, $(v,p)$ is analytic in~$\Om$.
\end{theorem}

We believe that the same strategy we use in this paper can be generalized for some nonlinear functions $g$ in Equation \eqref{E.Deg}, which can be engineered to produce interesting nontrivial solutions to the incompressible Euler equations. We will explore this direction in the future.

We conclude this subsection with two side remarks. First, the Schiffer conjecture had already appeared in connection with stationary Euler flows in an unpublished note of Kelliher~\cite{Kelli}, where he showed that a bounded domain $\Om\subset\RR^2$ admits a (nonconstant) Neumann eigenfunction that is constant on the boundary if and only if there is a stationary solution of the incompressible Euler equations on~$\Om$ that is also an eigenfunction of the Stokes operator on this domain, with the no-slip boundary condition. 

Second, recall that compactly supported 3D Euler flows are very different from their two-dimensional counterparts. In three dimensions, the existence of smooth stationary solutions of compact support was  established only recently in a groundbreaking paper by Gavrilov~\cite{Gavrilov}. These solutions have been revisited and put
in a broader context by Constantin, La and Vicol~\cite{CV} using the Grad--Shafranov formulation, which allowed them to obtain nontrivial applications
to other equations of fluid mechanics. Closer in spirit to the present paper is the construction of large families of discontinuous but piecewise smooth axisymmetric stationary Euler flows via overdetermined boundary value problems which was carried out in~\cite{ARMA}. The argument there does not rely on a bifurcation argument, but on an existence result for overdetermined problems where the existence of solutions (radial or not) is not known a priori, roughly along the lines of~\cite{PS:AIF, DS:DCDS, overdet,APDE}.

\subsection*{Strategy of the proof of Theorem \ref{T.main}}

Theorem~\ref{T.main} relies on a local bifurcation argument to construct nonradial solutions to the \textit{overdetermined} problem
\begin{equation} \label{problem}
\Delta u + \mu u = 0\quad \textup{in}\ \Omega\,, \quad \nabla u = 0 \quad \textup{on} \ \partial\Omega\,.
\end{equation}
%}
%
%nonradial Laplace eigenfunctions $u$ which satisfy the {\em overdetermined}\/ boundary condition $\nabla u=0$ on the boundary of some domain $\Om$. 
We start off with a parametrized family of radial Neumann eigenfunctions in annuli $\Om_a= \{a<r<1\}$, which are clearly solutions to \eqref{problem} when $\Om = \Om_a$. One should think about them as a family of ``trivial'' solutions. The gist of the proof is to show that there exists a family of ``nontrivial'' (i.e., nonradial)  solutions $u_s$ and domains $\Om_s$ satisfying \eqref{problem}, which branch out of the family of trivial solutions at a certain value~$a_*$ of the parameter $a$: these are the eigenfunctions and domains described in the statement of our main theorem. 

A bifurcation theorem, such as the Crandall--Rabinowitz theorem that we use in this paper, must specify a set of sufficient conditions which ensures that a family of nontrivial solutions does exist. Although checking that these conditions indeed hold is often hard, this kind of arguments have been successfully used to show that many overdetermined problems are {\em flexible}, in that they admit ``nontrivial'' solutions. For instance, solutions to overdetermined problems bifurcating from cylinders and strips, and from strips around the equator of the sphere have been found several times in the literature \cite{FMW2, FMW3, RSW1, SS12, S10}. %From these results it becomes apparent that these settings have a certain degree of flexibility for finding nontrivial solutions.

%These kind of solutions were first constructed on the cylinder by one of us in [9], and extended to other problems on the sphere and the cylinder in [3, 4], showing, for example, that Serrin’s symmetry results do not hold in these settings.

The essential property of these problems, however, is that being able to verify the hypotheses of a bifurcation theorem is not merely a technical problem: in fact, many important classes of overdetermined problems are known to be \textit{rigid} in the sense that there are no nontrivial solutions, see for instance to \cite{Sirakov,Serrin,Reichel}. The deeply geometrical interplay between rigidity and flexibility is perhaps the most distinctive feature of the study of overdetermined problems, and underlies Schiffer's conjecture about Neumann eigenfunctions. As a prime example of this dychotomy, note that while Serrin's symmetry result~\cite{Serrin} ensures that the only positive solutions to many overdetermined problems on a bounded domain of $\RR^n$ are radially symmetric, nontrivial solutions do bifurcate from radially symmetric ones in the case of periodic unbounded domains~\cite{FMW2, FMW3}.
%	\st{To put it differently, overdetermined problems are typically rigid on bounded domains on the plane (or on the upper half-sphere) but they can exhibit flexibility on unbounded periodic domains such us the cylinder (or on spherical domains which contain the equator).  }

From a conceptual point of view, our key contribution in this paper is to identify a novel geometric setting in which \eqref{problem} exhibits some flexibility: annular domains in the plane. This is the first flexibility result for an eigenvalue problem under overdetermined boundary conditions in bounded domains of the Euclidean space.  Looking for nontrivial solutions in this setting involves a leap of faith: this kind of domains were completely uncharted territory in the context of Schiffer-type problems and, contrarily to the kind of domains considered in~\cite{Fall-Minlend-Weth-main}, there were no indications that flexibility was to be expected. What is known, in fact, is that annular domains satisfy some fundamental partial rigidity properties~\cite{Reichel} that do not hold in the case of domains in the cylinder or large domains in the sphere. 

Although we eventually show that the new geometric setting of planar annuli has good flexibility properties, the strong rigidity properties it nonetheless exhibits turn out to be very important too. The ``partial rigidity'' of the annular domains we take in our paper is reflected in the strong rigidity properties of Neumann eigenfunctions on bounded Euclidean domains that are locally constant on the boundary: morally, this is why the problem we solve has the same known rigidity properties as the Schiffer conjecture, as we prove in Section~\ref{S.Rigidity} of our paper.

From a technical point of view, these rigidity properties are reflected in a number of new difficulties that arise naturally in our setting. To explain this fact, let us next sketch the proof of Theorem \ref{T.main}. As already mentioned, Theorem \ref{T.main} relies on a local bifurcation argument. To implement the bifurcation argument, we pass via a diffeomorphism to a fixed domain  $\{\frac12 < r < 1\}$, so that  our main unknown is a two-variable function $v(r,\te)$ defined on that annulus. This function encodes the shape of the domain $\Om$ and also of the solution $u$. To control the regularity of $v$, inspired by \cite{Fall-Minlend-Weth-main}, we use anisotropic H\"older spaces, since the functions arising are one derivative smoother in $r$ than they are in $\te$. Specifically, we consider suitable subspaces of the anisotropic space 
$$
\cX^{2,\al} := \big\{ u \in C^{2,\al}(\overline{\Om}_\frac12): \pd_r u \in C^{2,\al}(\overline{\Om}_\frac12) \big\}\,.
$$

Fall, Minlend and Weth introduced this space to  avoid a loss of derivatives phenomenon that appears when using a more standard functional setting. In this way, they construct the first examples of domains in the sphere that admit a Neumann eigenfunction which is constant on the boundary and whose boundary has nonconstant principal curvature, disproving a conjecture by Souam \cite{Souam}. In~\cite{Fall-Minlend-Weth-main}, a fundamental idea is to consider the overdetermined problem on domains that are invariant with respect to a reflection symmetry (either on the cylinder or on the sphere, in the second case domains that are symmetric across the equator). Conceptually, the reason is that this identifies a geometric setting in which one expects flexibility, which eventually enables them to bifurcate from any nonconstant radial Neumann eigenfunction. From a technical point of view, the reflection symmetry allows to reduce the analysis on the two connected components of the boundary to just one, and this reduction is  essential in their analysis.

In our geometric setting, we aim to prove bifurcation from an annulus in the plane, so there is no way to connect the two boundary components with a reflection. This qualitative change in the picture has the key technical effect that both connected components of the boundary are independent unknowns. This completely changes the setting of the problem, and its consequences affect the whole analysis. 

A first essential effect, which reflects the aforementioned additional rigidity of the problem, is that symmetry results imply there is no way one can bifurcate from any nonconstant Neumann eigenfunction in our setting, contrary to what happens in~\cite{Fall-Minlend-Weth-main}. In the bifurcation argument, this is seen because bifurcation can only occur when a suitable radial Neumann eigenvalue and a suitable nonradial Dirichlet eigenvalue coincide. In our case, this coincidence happens if the radial Neumann eigenvalue is the second (but not the first) and if the Fourier mode~$l$ of the Dirichlet eigenvalue satisfies $l \geq 4$.  This coincidence is proved in Proposition \ref{P.cross1}, which also ensures that there are no resonances with other Dirichlet eigenvalues, as required in the Crandall--Rabinowitz theorem. The inherent rigidity of the problem underlying this difficulty is addressed in Theorem \ref{T.Aviles}.

On the other hand, as can be expected in any application of the Crandall--Rabinowitz theorem, the hardest technicalities arise in the proof of the so-called ``transversality condition''. In our case, it is obtained as a corollary of Proposition \ref{P.cross2}, which is established for sufficiently large $l$. For completeness, in Appendix \ref{A.transversality}, we also check this condition numerically in the particular case $l = 4$. The proof of Proposition \ref{P.cross2} boils down to the asymptotic analysis of Dirichlet and Neumann eigenfunctions and eigenvalues on annuli, in presence of a large parameter (the angular mode of certain eigenfunction) and a small one (the thickness of the annulus). This analysis is rather subtle because these parameters are not independent, but related by the requirement that a certain pair of eigenvalues must coincide. To prove the required second order expansion of both eigenfunctions of eigenvalues, we combine ODEs methods with Fourier analysis. Actually, a considerable part of our paper is devoted to proving Proposition \ref{P.cross2}. One should note that using $\mathbb{Z}_l$-symmetric eigenfunctions (which allows us to only consider eigenfunctions involving Fourier modes of the form $\cos(ml\theta)$) is a useful technical trick, which has been used several times in the literature, but it does not conceptually change the geometric setting we are dealing with.

Once the transversal crossing of the eigenvalue branches has been established in Propositions \ref{P.cross1} and \ref{P.cross2}, and using the functional framework described above and presented in Section \ref{S.deform}, the proof of Theorem \ref{T.main} follows from the Crandall--Rabinowitz theorem, as we show in Section \ref{S.bifurcation} (Proposition \ref{P.CR} and Theorem \ref{T.mainDetails}).

\subsection*{Organization of the paper}

In Section~\ref{S.eigen} we prove the asymptotic estimates for the Dirichlet and Neumann eigenfunctions on annuli that are needed to establish eigenvalue crossing, non-resonance and transversality conditions. As discussed above, these will play a key role in the bifurcation argument later on. In Section \ref{S.deform} we reformulate the eigenvalue problem for a deformed annulus in terms of functions defined on a fixed domain and introduce the functional setting that we use. The proof of Theorem~\ref{T.main} is presented in Section~\ref{S.bifurcation}. The applications to an analog of the Pompeiu problem and to the construction of non-radial stationary Euler flows on the plane are developed in Section~\ref{S.corollaries}.  For completeness, in Section~\ref{S.Rigidity} we establish two rigidity results for the problem considered in this paper. We first show how existing arguments on the Schiffer problem prove that if a domain has infinitely many Neumann eigenfunctions that are locally constant on the boundary, then such domain is either a disk or an annulus. Then, we show that if the Neumann eigenvalue~$\mu$ is sufficiently low, then $\Om$ must be as well either a disk or an annulus. The paper is concluded with two appendices. In Appendix \ref{A.transversality}, using some elementary numerical computations, we show that the transversality condition holds for $l = 4$. Appendix~\ref{A.PullBack} contains some auxiliary computations that are used in the proof of the main theorem.

\section{The Dirichlet and Neumann spectrum of an annulus}
\label{S.eigen}

In this section we shall prove the auxiliary results about Dirichlet and Neumann eigenvalues of annular domains that we will need in the proof of Theorem~\ref{T.main}. Specifically, the main results of this section are Proposition~\ref{P.cross1}, where we establish the key eigenvalue crossing condition and a non-resonance property, and Proposition~\ref{P.cross2}, a technical result which will readily imply the  transversality condition required in the Crandall--Rabinowitz theorem. 

In view of the scaling properties of the Laplacian and of its invariance under translations, it suffices to consider annuli centered at the origin and with fixed outer radius.  In polar coordinates $(r,\te)\in\RR^+\times\TT$, let us therefore consider the annular domain
\[
\Om_a :=\{(r,\te)\in(a,1)\times\TT\}\,,
\]
where $a\in(0,1)$ and $\TT:=\RR/2\pi\ZZ$. 
It is well known that an orthogonal basis of $L^2(\Om_a,r \,dr \,d\te)$ of Neumann eigenfunctions in~$\Om_a $:
\[
\{\psi_{0,n}(r), \psi_{l,n}(r)\cos l\te,\psi_{l,n}(r)\sin l\te: l\geq1, \; n\geq0\}\,.
\]
Here, for each~$\leq 0$, $\{\psi_{l,n}\}_{n=0}^\infty$ is an orthonormal basis of $L^2((a,1),r\, dr)$ consisting of eigenfunctions of the associated radial operator. Thus these functions (whose dependence on~$a$ we omit notationally) satisfy the ODE
\begin{equation}\label{E.ODE.Neumann} \tag{${\rm N}_a^l$}
\psi_{l,n}''+\frac{\psi_{l,n}'}r-\frac{l^2 \psi_{l,n}}{r^2}+\mu_{l,n}(a)\, \psi_{l,n}=0\  \text{ in } (a,1)\,,\ \  \psi_{l,n}'(a)=\psi_{l,n}'(1)=0\,,
\end{equation}
for some nonnegative constants
\[
\mu_{l,0}(a)<\mu_{l,1}(a)<\mu_{l,2}(a)<\cdots < \mu_{l,n}(a) <\, \cdots 
\]
tending to infinity as $n\to\infty$. 
We will also omit the dependence of the eigenvalues on~$a$ when no confusion may arise. 

The Neumann spectrum of the annulus~$\Om_a $, counting multiplicities, is then $\{\mu_{l,n}\}_{l,n=0}^\infty$. Let us record here the following immediate properties of these eigenvalues:
$$ \mu_{0,0}=0  \quad \textup{ and } \quad \mu_{l,n} < \mu_{l',n} \ \mbox{ if } l < l'\,.
$$
Although the eigenvalues of the ODE~\eqref{E.ODE.Neumann} are all simple, the eigenvalues of~$\Om_a $ may have higher  multiplicity larger if $\mu_{l,n}=\mu_{l',n'}$ for some $(l,n)\neq (l',n')$. 

Likewise, an orthogonal basis of $L^2(\Om_a,r \,dr\, d\te)$ consisting of Dirichlet eigenfunctions in the annulus~$\Om_a $ is  
\[
\{\vp_{0,n}(r),\vp_{l,n}(r)\cos l\te,\vp_{l,n}(r)\sin l\te\}_{l=1,n=0}^\infty\,,
\]
where now the radial eigenfunctions satisfy the ODE
\begin{equation}\label{E.ODE.Dirichlet} \tag{${\rm D}_a^l$}
\vp_{l,n}''+\frac{\vp_{l,n}'}r-\frac{l^2 \vp_{l,n}}{r^2}+\la_{l,n}(a)\, \vp_{l,n}=0\ \text{ in } (a,1)\,,\ \ \vp_{l,n}(a)=\vp_{l,n}(1)=0\,,
\end{equation}
for some positive constants
\[
\la_{l,0}(a)<\la_{l,1}(a)<\la_{l,2}(a)<\cdots < \la_{l,n}(a) <\, \cdots
\]
tending to infinity as $n \to \infty$. The Dirichlet spectrum of~$\Om_a $ is therefore $\{\la_{l,n}\}_{l,n=0}^\infty$. Again, we have
$$ \lambda_{l,n} < \lambda_{l',n} \ \mbox{ if } l < l'\,.$$

Observe also that, thanks to the min-max characterization of the eigenvalues, the Neumann and Dirichlet eigenvalues must satisfy$$ \mu_{l,n} < \lambda_{l,n}\,.$$

In the following elementary lemma we show a property relating radial Neumann eigenvalues with Dirichlet ones. This is a variation on the fact that the derivative of a (nonconstant) radial Neumann eigenfunction is a Dirichlet eigenfunction.

\begin{lemma} \label{extra} For any $n \geq 1$, $\mu_{0,n} = \lambda_{1,n-1}.$ \end{lemma}

\begin{proof}
Recall that $\psi_{0,n}$ is a solution to
$$
\psi_{0,n}''+\frac{\psi_{0,n}'}r+\mu_{0,n}\, \psi_{0,n}=0\quad \text{in } (a,1)\,,\qquad \psi_{0,n}'(a)=\psi_{0,n}'(1)=0\,.
$$
Moreover, $\psi_{0,n}$ has exactly $n$ changes of sign in the interval $(a,1)$. As a consequence, $\phi= \psi_{0,n}'$ changes sign $n-1$ times. Differentiating the equation above, we obtain that $\phi$  solves
$$
\phi''+\frac{\phi'}r-\frac{\phi}{r^2}+\mu_{0,n}\, \phi=0\quad \text{in } (a,1)\,, \qquad \phi(a)=\phi(1)=0\,.
$$
Since the solution space is one-dimensional, $\phi = \vp_{1,n-1}$ (possibly up to a nonzero multiplicative constant), so the lemma is proved. 	
\end{proof}

Next we show that there exist annuli where the third radial Neumann eigenvalue and the first $l$-mode Dirichlet eigenvalue coincide, and where certain non-resonance conditions are satisfied. This technical result is the first key ingredient in the bifurcation argument that the proof of Theorem~\ref{T.main} hinges on.

\begin{proposition}\label{P.cross1}
	For any integer $l \geq 4$, there exists some constant $a_l \in (0,1)$ such that 
	\begin{equation} \label{id.eigenvaluesBif}
		\mu_{0,2}(a_l)=\la_{l,0}(a_l)\,.
	\end{equation}
Moreover, the following non-resonance condition holds:
\begin{equation} \label{NR} \tag{NR} \la_{l,0}(a_l) \neq \lambda_{ml,n}(a_l) \quad \textup{ for all }\quad (m,n)\neq (1,0).\end{equation}
\end{proposition}

\begin{proof}

We set $h := \tfrac{1-a}{\pi}$ and $r =: 1-hx$. Suppose now that $h\ll1$, so $a$ is very close to~1. By the variational characterization of $\MU$, we get
\begin{equation} \label{E.asymp1}
\begin{aligned}
	\MU(a) & = \inf_{\dim V = 3} \, \max_{\psi \in V \setminus \{0\}} \frac{\displaystyle \int_0^{\pi} \psi'(x)^2(1-hx)\, dx}{h^2 \displaystyle \int_0^\pi \psi(x)^2(1-hx)\, dx}\\
	& = \big[h^{-2}+O(h^{-1}) \big]  \inf_{\dim V = 3} \, \max_{\psi \in V \setminus \{0\}} \frac{\displaystyle \int_0^{\pi} \psi'(x)^2\, dx}{\displaystyle \int_0^\pi \psi(x)^2\, dx} \\[0.25cm]
	& = 4h^{-2}+O(h^{-1})\,.
\end{aligned}
\end{equation}
Here, $V$ ranges over the set of 3-dimensional subspaces of $H^1((0,\pi))$ and, in the last identity, we are using that
\begin{equation} \label{E.Eigen4}
 \inf_{\dim V = 3} \, \max_{\psi \in V \setminus \{0\}} \frac{\displaystyle \int_0^{\pi} \psi'(x)^2\, dx}{\displaystyle \int_0^\pi \psi(x)^2\, dx}  = 4\,.
\end{equation}
This follows from the fact that this is the Courant--Fischer formula for the second nontrivial Neumann eigenvalue of $\frac{\rm d^2}{{\rm d}x^2}$ in $(0,\pi)$. As is well known, the corresponding eigenfunction is $\psi(x) := \cos(2x)$, and the eigenvalue is $4$.
\medbreak
Likewise, in the case of the first $l$-mode Dirichlet eigenvalue, we have
 
\begin{align}
	\la_{l,0}(a) & = \inf_{\varphi \in H_0^1((0,\pi)) \setminus\{0\}}  \, \frac{\displaystyle h^{-2} \int_0^{\pi} \psi'(x)^2(1-hx)\, dx + l^2 \int_0^{\pi} \frac{\psi(x)^2}{(1-hx)}\, dx}{ \displaystyle \int_0^\pi \psi(x)^2(1-hx)\, dx} \nonumber \\ \label{E.asymp2}
	 & = h^{-2} \left [  \inf_{\varphi \in H_0^1(0,\pi) \setminus\{0\}}  \, \frac{ \displaystyle \int_0^{\pi} \psi'(x)^2\, dx} { \displaystyle \int_0^\pi \psi(x)^2\, dx} + O(h) \right ] + l^2 (1+O(h))\\[0.5cm] \nonumber
	 & = \big(1+(hl)^2 \big) \big[h^{-2}+O(h^{-1})\big]\,.
\end{align}
Note that in the last identity we are using that 
$$ \inf_{\varphi \in H_0^1(0,\pi) \setminus\{0\}}
\frac{ \displaystyle \int_0^{\pi} \psi'(x)^2\, dx} { \displaystyle \int_0^\pi \psi(x)^2\, dx} = 1\,,
$$
which follows the characterization of the first Dirichlet eigenvalue of $\frac{\rm d^2}{{\rm d}x^2}$ in $(0,\pi)$ using the Courant--Fischer formula. Of course, the first Dirichlet eigenvalue of  $\frac{\rm d^2}{{\rm d}x^2}$ in $(0,\pi)$ is $\psi(x):= \sin(x)$, with  eigenvalue~$1$. 

%Hence, it follows that
%\begin{equation}
%	\lambda_{l,0}(a) = \big(1+(hl)^2 \big) \big[h^{-2}+O(h^{-1})\big]\,.
%\end{equation}
In particular, we conclude from these expressions for $\MU$ and $\lambda_{l,0}$ that 
\begin{equation}  \label{E.muBiggerLambda}
	\MU(a) > \lambda_{l,0}(a)
\end{equation}
whenever $a$ is close enough to $1$.

We now show the converse inequality for $a$ small. In order to do that, recall that, by Lemma \ref{extra}, $\mu_{0,2}(a)= \la_{1,1}(a)$. We now claim that for any $l \geq 1$, $k \geq 0$, 
\begin{equation} \label{disco}
\la_{l,k}(a) \to \la_{l,k}^0\,, \quad \textup{ as } a \to 0\,,
\end{equation}
where we have set
\begin{equation*}  
\la_{l,k}^0:= \inf_{\dim V = k} \, \max_{\varphi \in V \setminus \{0\}} \frac{\displaystyle \int_0^{1} \Big( \varphi'(r)^2  + l^2 \frac{\varphi(r)^2}{r^2} \Big)\, r\, dr}{\displaystyle \int_0^1 \varphi(r)^2\,r\, dr}\,.
\end{equation*}
Here, $V$ ranges over the set of $k$-dimensional subspaces of $\mathcal{H}_0 \subset \mathcal{H}$, where
$$
\mathcal{H} := \Big\{\varphi: (0,1) \to \RR: \int_0^{1} \Big( \varphi'(r)^2  +  \frac{\varphi(r)^2}{r^2}\Big)\,r \, dr < +\infty\ \Big\}\,, $$
and
$$  \mathcal{H}_0:= \big\{\varphi \in \mathcal{H}:\ \varphi(1)=0\big\}\,.
$$
These spaces are equipped with the natural norm
$$
\| \varphi\|_{\mathcal{H}}^2 := \int_0^{1} \left(\varphi'(r)^2  +  \frac{\varphi(r)^2}{r^2}\right)\, r\, dr  \,.
$$
\noindent Note that the functions in $\mathcal{H}$ are continuous in $[0,1]$ and vanish at $0$, see e.g. \cite[Proposition 3.1]{byeon}.

In other words, \eqref{disco} simply asserts that the Dirichlet eigenvalues of the annulus $\Omega_a$ converge, as $a \to 0$, to the Dirichlet eigenvalues of the unit disk. This is well known, but we give a short self-contained proof for the convenience of the reader.  

First, observe that $\la_{l,k}(a)$ is increasing in $a$, and it is always bigger than~$\la_{l,k}^0$. As a consequence, it has a limit and
$$
 \lim_{a\to 0} \la_{l,k}(a) \geq \la_{l,k}^0\,.
$$

To show the converse inequality, it suffices to prove that any function in~$\mathcal{H}_0$ can be approximated by functions vanishing in a neighborhood of $0$. For this, let us choose a smooth cut-off function $\chi:[0, +\infty) \to [0,1] $ such that $\chi|_{[0,1]}=1$ and $\chi_{[2,+\infty)} =0$, and define $\chi_\ep (r):= \chi(\frac{r}{\ep})$ for $\ep $ sufficiently small. We shall next show that, given any $\varphi \in \mathcal{H}$, $\chi_{\ep} \varphi \to 0$ in $\mathcal{H}$ as $\ep \to 0$. This obviously implies~\eqref{disco} since in this case $(1-\chi_{\ep}) \varphi \to \varphi$ in $\mathcal{H}$. 

To show that $\chi_{\ep} \varphi \to 0$ in $\mathcal{H}$,  we first note that
$$ 
\int_0^1 \frac{\chi_\ep(r)^2 \varphi(r)^2}{r} \leq \int_0^\ep \frac{ \varphi(r)^2}{r} \to 0\,, \quad \textup{ as } \ep \to 0\,.
$$
Moreover, as $\ep\to0$,
$$
\int_0^1 \chi_\ep(r)^2 \varphi'(r)^2 r \, dr \leq \int_0^\ep \varphi'(r)^2 r \, dr \to 0\,,
$$
and
$$ 
\begin{aligned}
\int_0^1 \chi_\ep'(r)^2 \varphi(r)^2 r \, dr & \leq \frac{\| \chi'\|_{L^{\infty}}^2}{\ep^2} \int_0^\ep  \varphi(r)^2 r \, dr \\
& \leq \frac{\| \chi'\|_{L^{\infty}}^2}{\ep} \int_0^\ep  \varphi(r)^2 \, dr \leq\| \chi'\|_{L^{\infty}}^2\| \varphi\|_{L^{\infty}([0,\ep])}^2 \to 0\,,
\end{aligned}
$$
since $\varphi$ is continuous in $[0,1]$ and $\varphi(0)=0$. As $(\chi_{\ep} \varphi)'= \chi_{\ep}' \varphi + \chi_{\ep} \varphi'$, it follows that $\chi_{\ep} \varphi \to 0$ in $\cH$ and thus \eqref{disco} follows. 
 
We now compute the eigenvalues $\lambda_{1,1}^0$ and $\la_{l,0}^0$, which can be written explicitly in terms of Bessel zeros. More precisely, it is known that if $j_{m,k}$ is the $k$-th positive zero of the Bessel function of the first kind~$J_m$, then
\[
\lambda_{l,k}^0=j_{l,k+1}^2\,.
\]
The numerical values of the zeros are well known, e.g.,
$$
 j_{1,2} = 7.01559\dots
$$
and
$$
j_{1,1} =  3.83171\dots\,, \  j_{2,1} = 5.13562\dots\,, \  j_{3,1} =  6.38016\dots \,,\  j_{4,1} =  7.58834\dots
$$
%\begin{center} \begin{tabular}{| c | c |}
%	\hline    $l$ $ $  &  $j_{l,1}$ \\
%	\hline	1 &   $ \quad 3.83171\dots \quad$\\
%	2 &      $\quad 5.13562\dots \quad$\\
%	3 &   $\quad 6.38016\dots \quad$ \\
%	4 &    $\quad 7.58834\dots \quad$ \\
%	\hline
%\end{tabular} \end{center}

\begin{figure}[h]
	\centering 
	\begin{minipage}[c]{105mm}
		\centering
		\resizebox{105mm}{62mm}{\includegraphics{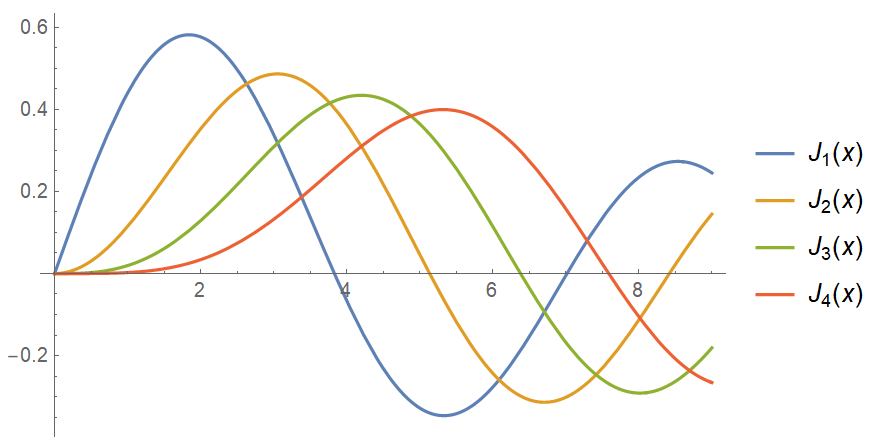}}
	\end{minipage}
	\caption{The first positive zero of $J_l$ is bigger than the second positive zero of $J_1$ if $l \geq 4$.}
\end{figure}

Also, note that $j_{l,k+1} < j_{l',k+1} \textup{ if } l < l'$. Therefore, $\lambda_{1,1}^0 = j_{1,2}^2 < j_{l,1}^2 = \la_{l,0}^0$ for any $\l \geq 4$. By Lemma \ref{extra} and \eqref{disco} we obtain that, for any integer $l \geq 4$ and for sufficiently small $a$, $\mu_{0,2}(a) = \la_{1,1}(a) < \la_{l,0}(a)$. 
This inequality, together with \eqref{E.muBiggerLambda} and the continuous dependence of the eigenvalues on $a$ (see e.g. \cite[Section II.7]{Kato}), allows us to conclude that, for any integer $l \geq 4$, there exists $a_l\in(0,1)$ such that \eqref{id.eigenvaluesBif} holds.
 
We now turn our attention to \eqref{NR}. In the rest of the proof, we fix $a:= a_l$, omitting the dependence on this variable for the ease of notation. First of all, by the strict monotonicity with respect to~$l$ and~$n$ of the Neumann and Dirichlet eigenvalues, it follows that
\begin{align*}
	 \la_{l, n} > \la_{l,0} \ \textup{ for all }n \geq 1\,, \qquad 
	\lambda_{ml,n} \geq \lambda_{ml,0}> \LA \ \textup{ for all } m \geq 2, \ n \geq 0\,.
	\end{align*}
Moreover,
$$
 \la_{0,0} < \la_{l,0}\,, \qquad \lambda_{0, n} \geq \lambda_{0,2} > \mu_{0,2} = \lambda_{l,0} \ \mbox { if } n \geq 2\,.
$$
Finally, by Lemma \ref{extra}, $\mu_{0,2} = \lambda_{1,1}$, which is clearly bigger than $ \lambda_{0,1}$. The proposition is then proved.
\end{proof}

Next proposition is devoted to prove an estimate which will imply the usual transversality condition in the Crandall--Rabinowitz theorem. We prove this estimate for any sufficiently large $l$, and refer to Appendix~\ref{A.transversality} for the direct verification of this condition in the case $l=4$ (the same strategy could be used for other specific values of~$l$). 
In the statement, we denote by a prime the derivative of an eigenvalue with respect to the parameter~$a$.

\begin{proposition}\label{P.cross2}
For any large enough integer $l$, the constant $a_l$ introduced in Proposition \ref{P.cross1} is of the form
\begin{equation} \label{E.al}
	a_l = 1-\frac{\sqrt{3}\pi}{l}+ O(l^{-2})\,.
\end{equation}
	Moreover, the following transversality condition holds:
\begin{equation} \label{T} \tag{T} \mu_{0,2}'(a_l) \neq \lambda_{l,0}'(a_l)\,. \end{equation}
\end{proposition}

\begin{proof}

The proof will be divided into four steps for the sake of clarity.

\subsubsection*{Step 1: Zeroth order estimates} 

In this first step we prove \eqref{E.al}. First of all, observe that
\begin{equation}\label{E.easylal}
\lambda_{l,0}(a) := \inf_{\vp\in H_0^1((a,1))\setminus \{0\}} \displaystyle \frac{\displaystyle\int_a^1 \Big(\vp'(r)^2 + \frac{l^2}{r^2} \vp(r)^2 \Big)\, r\, dr}{\displaystyle\int_a^1 \vp(r)^2 r \, dr} \geq l^2\,, 
\end{equation}
for all $a \in (0,1).$ This implies in particular that, for any sufficiently large integer $l$, 
\begin{equation} \label{E.lambdaBiggerMu}
	\lambda_{l,0}(\tfrac12) > \MU(\tfrac12)\,.
\end{equation}

As a consequence, we can take $a_l \in (\tfrac12, 1)$ for sufficiently large $l$. As in the proof of Proposition \ref{P.cross1}, we define $h_l := \tfrac{1-a_l}{\pi}$ and $r =: 1-h_lx$. Then, we can estimate
\begin{align*}
\MU(a_l) & = \inf_{\dim V = 3} \, \max_{\psi \in V \setminus \{0\}} \frac{\displaystyle \int_0^{\pi} \psi'(x)^2(1-h_lx)\, dx}{h_l^2 \displaystyle \int_0^\pi \psi(x)^2(1-h_lx)\, dx}  \\
& \leq  \inf_{\dim V = 3} \, \max_{\psi \in V \setminus \{0\}} \frac{2\displaystyle \int_0^{\pi} \psi'(x)^2\, dx}{\displaystyle h_l^2\int_0^\pi \psi(x)^2\, dx} = 8h_l^{-2}
\end{align*}
which, together with \eqref{E.easylal}, implies that $h_l = O(l^{-1})$ for large~$l$. Note that in the inequality we have used that $a_l\in(\frac12,1)$, so $h_l\leq 2/\pi$, and in the last identity we have again used \eqref{E.Eigen4}. Hence, we can use the asymptotic formulas~\eqref{E.asymp1} and \eqref{E.asymp2} to infer that
$$
4 + O(h_l)= 1+ (h_ll)^2\,.
$$
Thus, \eqref{E.al} holds.

%\begin{remark}
%Refining the argument a little, one can show that for every $l\geq4$ there exists some $a_l \in (0,1)$ for which \eqref{id.eigenvaluesBif} holds. However, in later steps of the proof we will need more refined information about the eigenvalues which we can only obtain for $l\gg1$.
%\end{remark}

\subsubsection*{Step 2: First order estimates and asymptotics for~$a_l$}

Our goal now is to obtain a more accurate asymptotic estimate for $a_l$ in terms of~$l\gg1$. To this end, let us introduce the differential operator
\begin{equation} \label{E.TetaEpsilon}
T_{\eta,\varepsilon} := \partial_x^2 - \frac{\eta(1+\varepsilon)}{1-\eta(1+\varepsilon)x}\, \partial_x - \frac{3(1+\ep)^2}{(1-\eta(1+\ep)x)^2} \,,
\end{equation}
which naturally arises when we consider the eigenvalue equations, \eqref{E.ODE.Neumann} and \eqref{E.ODE.Dirichlet}, in the $x$-variable given by the relation $r =: 1-hx$, with $\eta := \frac{\sqrt{3}}{l}$ and $h = \frac1{\pi}(1-a) =: \eta(1+\ep)$.

Here $(\eta,\varepsilon)$ are real constants, which can be assumed to be suitably small (say, $|\eta| + |\epsilon| < \tfrac{1}{100}$). Note that for $\eta=\varepsilon=0$, the operator is simply $T_{0,0}=\pd_x^2-3$.

Let us respectively denote by  $\Lambda_n(\eta,\ep)$ and ${\nu}_n(\eta,\ep)$ the Dirichlet and Neumann eigenvalues of the operator~$T_{\eta,\ep}$ on the interval $(0,\pi)$, where $n$ ranges over the nonnegative integers. It is standard (and easy to see) that all these eigenvalues have multiplicity~1, so it is well known~\cite[Section II.7]{Kato}
that the eigenvalues and their corresponding eigenfunctions (which we will respectively denote by $\Phi_n(\eta,\ep)$ and $\Psi_n(\eta,\ep)$) depend analytically on~$(\eta,\ep)$. Therefore, for all $n \geq 0$, we can assume that 
\begin{equation} \tag{${\rm N}_{\eta,\ep}$} \label{E.NeumannTep} 
T_{\eta,\ep} \Psi_n + \nu_n(\eta,\ep) \Psi_n = 0 \quad \textup{in } (0,\pi)\,, \qquad \Psi_n'(0) = \Psi_n'(\pi) = 0\,,
\end{equation}
and
\begin{equation} \tag{${\rm D}_{\eta,\ep}$} \label{E.DirichletTep}
T_{\eta,\ep} \Phi_n + \Lambda_n(\eta,\ep) \Phi_n = 0 \quad \textup{in } (0,\pi)\,, \qquad \Phi_n(0) = \Phi_n(\pi) = 0\,.
\end{equation}

Likewise, to deal with the $0$-mode case, we introduce the operator
\begin{equation} \label{E.TtildeEtaEpsilon}
\widetilde{T}_{\eta,\ep}:= \partial_x^2 - \frac{\eta(1+\varepsilon)}{1-\eta(1+\varepsilon)x}\, \partial_x\,,
\end{equation}
and, for all nonnegative integer $n$, we denote by $\widetilde{\nu}_n(\eta,\ep)$ and $\widetilde{\Lambda}_n(\eta,\ep)$ the Neumann and Dirichlet eigenvalues in $(0,\pi)$ with corresponding eigenfunctions $\widetilde{\Psi}_n(\eta,\epsilon)$ and $\widetilde{\Phi}_n(\eta,\ep)$ respectively. In other words, 
\begin{equation} \tag{${\rm \widetilde{N}}_{\eta,\ep}$}
\widetilde{T}_{\eta,\ep} \widetilde{\Psi}_n + \widetilde{\nu}_n(\eta,\ep) \widetilde{\Psi}_n = 0 \quad \textup{in } (0,\pi)\,, \qquad \Psi_n'(0) = \Psi_n'(\pi) = 0\,,
\end{equation}
and 
\begin{equation} \tag{${\rm \widetilde{D}}_{\eta,\ep}$}
\widetilde{T}_{\eta,\ep} \widetilde{\Phi}_n + \widetilde{\Lambda}_n(\eta,\ep) \widetilde{\Phi}_n = 0 \quad \textup{in } (0,\pi)\,, \qquad \Phi_n(0) = \Phi_n(\pi) = 0\,.
\end{equation}

Let us emphasize that the reason for introducing these operators is that one can conveniently express $\mu_{0,2}(a_l)$ and $\la_{l,0}(a_l)$ in terms of the eigenvalues $\widetilde{\nu}_2$ and $\La_0$ of the operators $T_{\eta,\ep}$ and $\widetilde{T}_{\eta,\ep}$, respectively. Indeed, if we take $\eta_l = \frac{\sqrt{3}}{l}$ and write $a_l = 1-\pi h_l =: 1- \pi \eta_l (1+\ep_l)$, using the change of variables $r =: 1-h_l x$, a direct calculation shows that the radial eigenvalue equation $($N$_{a_l}^0)$ is 
\begin{equation*}
\widetilde{T}_{\eta_l,\ep_l} u + h_l^2 \mu_{0,n}(a_l) u =0 \quad \textup{in } (0,\pi)\,, \qquad u'(0) = u'(\pi) = 0\,.
\end{equation*}
Likewise, the radial eigenvalue equation $($D$_{a_l}^{l})$ reads as
\begin{equation*}
T_{\eta_l, \ep_l} u + h_l^2 \lambda_{l,n}(a_l) u = 0 \quad \textup{in } (0,\pi)\,, \qquad u(0) = u(\pi) = 0\,.
\end{equation*}
We then infer that
\begin{equation} \label{E.relationEigenvalues}
\begin{aligned}
\MU(a_l) & = h_l^{-2} \, \nuT_2(\eta_l,\ep_l)\,, \\
\la_{l,0}(a_l) & = h_l^{-2}\, \La_0(\eta_l,\ep_l)\,.
\end{aligned}
\end{equation}

Keeping these identities  in mind, we start with the proof of first order asymptotics for $\Lambda_0$ and $\widetilde{\nu}_2$:
 
\begin{lemma} \label{L.1stOrderLambdaMu}
For $|\eta|+|\ep| \ll 1$,
\begin{align*}
\Lambda_0(\eta,\ep) & = 4 + 3\pi \eta + 6 \ep + O(\eta^2+\ep^2)\,,\\
\widetilde{\nu}_2(\eta,\ep) & = 4+ O(\eta^2+\ep^2)\,.
\end{align*}
\end{lemma}

\begin{proof}
Let us start with $\Lambda_0$. Noting that $\Phi_{00}(x):=\sin(x)$ is the first Dirichlet eigenfunction of $T_{0,0}=\pd_x^2-3$ on $(0,\pi)$ with eigenvalue $\La_{00}:=4$, by the analytic dependence on the parameters, we can take an analytic two-parameter family of Dirichlet eigenfunctions of the form
$$
\Phi_0(\eta,\ep)(x) := \sum_{j,k \geq 0} \eta^j \ep^k \, \Phi_{jk}(x)\,,  \quad \textup{with} \quad \Phi_{jk}(0) = \Phi_{jk}(\pi) = 0\,, \quad \forall\ j,k \geq 0\,,
$$
with  corresponding eigenvalues
$$
\Lambda_0(\eta,\ep) := \sum_{j,k \geq 0} \eta^j \ep^k \, \Lambda_{jk}\,.
$$
Moreover, let us normalize $\Phi_0(\eta,\ep)$ so that
\[
\int_{0}^{\pi} \Phi_0(\ep,\eta)(x) \sin(x)\, dx = \frac{\pi}{2}\,.
\]
Since $\int_{0}^{\pi} \Phi_{00}(x) \sin(x)\, dx = \frac{\pi}{2}$, this implies 
\[
\int_{0}^{\pi} \Phi_{jk}(x) \sin(x)\, dx = 0 \quad \textup{for all} \quad (j,k) \neq (0,0)\,.
\]

Let us now compute the first order terms in the expansion. Since
\begin{align*}
T_{\eta,\epsilon} =  &\ \partial_x^2 - \big(\eta+\ep\eta+\eta^2 x - O(\ep \eta^2 + \eta^3) \big) \,\partial_x \\
& - \big( 3+6\ep + 3 \ep^2 + 6 \eta x + 18 \eta \ep x + 9 \eta^2 x^2 + O(\eta^2 \ep + \ep^2 \eta + \eta^3 + \ep^3) \big)\,,
\end{align*}
defining $L:=  \partial_x^2 + \Lambda_{00} -3=\pd_x^2+1$, we get 
\begin{align*}
0 & = T_{\eta,\ep} \Phi_0 + \La_0(\eta,\ep) \Phi_0 \\
& = L \Phi_{00} + \ep \Big( L \Phi_{01} + ( \Lambda_{01} - 6 )\Phi_{00} \Big)\\
& \quad  + \eta \Big( L \Phi_{10} - \partial_x \Phi_{00} + (\Lambda_{10}-6x) \Phi_{00} \Big) + O(\eta^2+\ep^2)\quad \textup{in } (0,\pi)\,.
\end{align*}
Also, as $\Lambda_{00} = 4$ and $\Phi_{00}(x)= \sin(x)$, the zeroth order term vanishes, and we arrive at the equations 
\begin{align}\label{E.Phi01}
L \Phi_{01} + ( \Lambda_{01} - 6 )\Phi_{00} &= 0  \,, \\
	\label{E.Phi10}
L \Phi_{10} - \partial_x \Phi_{00} + (\Lambda_{10}-6x) \Phi_{00}&= 0 \,.
\end{align} 

Multiplying Equation~\eqref{E.Phi01} by $\Phi_{00}$ and integrating by parts, we easily find that $\Lambda_{01} = 6$ and $\Phi_{01} \equiv 0$. Doing the same with~\eqref{E.Phi10}, we get 
$$
\Lambda_{10} \, \frac{\pi}{2} = \Lambda_{10} \int_0^\pi \sin(x)^2\, dx = 6 \int_0^\pi x \sin(x)^2\, dx = \frac{3\pi^2}{2}\,,
$$
so $\Lambda_{10} = 3\pi$. We then conclude that
$$
\Lambda_0(\eta,\ep) = 4+3\pi\eta+6\ep + O(\eta^2+\ep^2)\,.
$$

We now compute the first order expansion of $\widetilde{\nu}_2$. We start with the observation that $\psiT_{00}(x):=\cos(2x)$ is the third Neumann eigenfunction of $\widetilde T_{0,0}=\pd_x^2$ in $(0,\pi)$, and that the corresponding eigenvalue is~$\widetilde{\nu}_{00}:=4$. Arguing as above, we consider a convergent series expansion of the eigenvalue
$$
\widetilde{\nu}_2(\eta,\ep) := \sum_{j,k \geq 0} \eta^j \ep^k\, \nuT_{jk}
$$
and of the corresponding eigenfunction
$$
\widetilde{\Psi}_2(\eta,\ep)(x):= \sum_{j,k\geq 0} \eta^j \ep^k\,  \psiT_{jk}(x)\,, \quad \textup{with} \quad  \psiT_{jk}'(0) = \psiT_{jk}'(\pi) = 0\,, \quad \forall \ j,k \geq 0\,,
$$
which we normalize so that
$$
\int_0^{\pi} \psiT_2(\eta,\ep)(x)\, \cos(2x) dx = \frac{\pi}{2}= \int_0^{\pi} \psiT_{00}(x)\, \cos(2x) dx\,.
$$

Since
$$
\widetilde{T}_{\eta,\ep} = \partial_x^2 - \big(\eta+\ep\eta+\eta^2 x + O(\ep \eta^2 + \eta^3) \big) \,\partial_x \,,
$$
just as in the case of $\Lambda_0$ we can write
\begin{align*}
0 & = \widetilde{T}_{\eta,\ep} \psiT_2 + \nuT_2(\eta,\ep) \psiT_2 \\
&  = \widetilde{L} \psiT_{00} + \ep \Big( \widetilde{L} \psiT_{01} + \nuT_{01} \psiT_{00} \Big) \\
& \quad + \eta \Big( \widetilde{L} \psiT_{10} - \partial_x \psiT_{00} + \nuT_{10} \psiT_{00} \Big) + O(\eta^2+\ep^2) \quad \textup{in } (0,\pi)\,,
\end{align*}
with $\widetilde{L}:= \partial_x^2 + \nuT_{00}=\pd_x^2+4$.
Arguing as we did in the analysis of $\Lambda_0$, multiplying the equations now by $\psiT_{00}$ and integrating by parts, we infer that $\nuT_{01} = \nuT_{10} = 0$, and thus we conclude that
$$
\nuT_2(\eta,\ep) = 4 + O(\eta^2+\ep^2)\,.
$$
The lemma is then  proved.
\end{proof}

Armed with this auxiliary estimate, we readily obtain the following second order formula for $a_l$:

\begin{lemma}  \label{L.epsilonl}
For any large enough integer $l$, the constant $l$ introduced in Proposition \ref{P.cross1} is of the form
\begin{equation} \label{E.secondOrderal}
a_l = 1-\frac{\sqrt{3}\pi}{l} + \frac{3\pi^2}{2l^2} + O(l^{-3})\,.
\end{equation}
\end{lemma}

\begin{proof}
The result follows from the condition $\LA(a_l)=\MU(a_l)$, \eqref{E.relationEigenvalues}, and the asymptotic formulas of Lemma~\ref{L.1stOrderLambdaMu}.
\end{proof} 

\begin{remark} \label{R.L.epsilonl}
For later convenience, let us write $
a_l = 1-\pi\eta_l (1+\ep_l)$,
with $\eta_l := \frac{\sqrt{3}}{l}$, and reformulate the expansion in \eqref{E.secondOrderal} as follows: there exists  $\delta_l = O(l^{-1})$ such that
\begin{equation} \label{E.epsilonl}
\ep_l = -\frac{\sqrt{3}\pi}{2l} \big(1 + \delta_l \big)\,.
\end{equation} 
\end{remark}

\subsubsection*{Step 3: Second order estimates}

In view of Lemma~\ref{L.epsilonl} and Remark \ref{R.L.epsilonl}, in the operators $T_{\eta,\ep}$ and $\widetilde T_{\eta,\ep}$ it is natural to take $\ep := -\frac\pi{2} \eta(1+\delta)$ with $|\delta| \ll 1$. By abuse of notation, we shall denote by
$$
T_{\eta,\de} := \partial_x^2 - \frac{\eta(1-\frac\pi{2} \eta(1+\delta))}{1-\eta(1-\frac\pi{2} \eta(1+\delta))x}\, \partial_x - \frac{3(1-\frac\pi{2} \eta(1+\delta))^2}{(1-\eta(1-\frac\pi{2} \eta(1+\delta))x)^2} \,,
$$
the resulting differential operator, and we shall use the notation $\nu_n(\eta,\de)$, $\Lambda_n(\eta,\de)$, $\Psi_n(\eta,\de)$ and $\Phi_n(\eta,\de)$ for its eigenvalues and eigenfunctions. Just as above, these operators will provide an efficient way of analyzing the asymptotic behavior of the eigenvalues.

Likewise, we define
$$
\widetilde{T}_{\eta,\de} := \partial_x^2 - \frac{\eta(1-\frac\pi{2} \eta(1+\delta))}{1-\eta(1-\frac\pi{2} \eta(1+\delta))x}\, \partial_x\,,
$$
and we use a similar notation for its eigenvalues and eigenfunctions: $\nuT_n(\eta,\de)$, $\LaT_n(\eta,\de)$, $\psiT_n(\eta,\de)$ and $\phiT_n(\eta,\de)$. Of course, the dependence on the new parameters $(\eta,\de)$ is still analytic.

The second order asymptotic expansions that we need are the following:

\begin{lemma} \label{L.2ndOrder}
For $|\eta|+|\de| \ll 1$,
\begin{align*}
\Lambda_0(\eta,\de) & = 4-3\pi\eta\de -16 \eta^2 + O(|\eta|^3+|\de|^3)\,, \\
%\nu_1(\eta,\de) & = 4-3\pi\eta\de + 24  \eta^2 + O(|\eta|^3+|\de|^3)\,, \\
\nuT_2(\eta,\de) & = 4 + \tfrac34\, \eta^2 + O(|\eta|^3+|\de|^3)\,.
%\LaT_1(\eta,\de) & = 4 + \tfrac72\,\eta^2 + O(|\eta|^3+|\de|^3)\,.
\end{align*}
\end{lemma}

\begin{proof}
We start with $\Lambda_0$. In the case $\eta=\de=0$, it is clear that $\Phi_{00}(x) := \sin(x)$ is an eigenfunction of $T_{\eta,\de}$ with eigenvalue $\Lambda_{00} := 4$. Hence, there exists an eigenfunction of the form
$$
\Phi_0(\eta,\de)(x) := \sum_{j,k \geq 0} \eta^j \de^k\, \Phi_{jk}(x)\,, \quad \textup{with} \quad \Phi_{jk}(0) = \Phi_{jk}(\pi) = 0\,, \quad \forall \ j,k \geq 0\,,
$$
with eigenvalue
\[
\Lambda_0(\eta,\de) := \sum_{j,k\geq 0} \eta^j \de^k\, \Lambda_{jk}\,,
\]
which we normalize as
$$
\int_{0}^{\pi} \Phi_0(\eta,\de)(x) \sin(x)\, dx = \frac{\pi}{2}= \int_{0}^{\pi} \Phi_{00}(x) \sin(x)\, dx\,,
$$
to ensure that
$$
\int_{0}^{\pi} \Phi_{jk}(x) \sin(x)\, dx = 0 \quad \textup{for all} \quad (j,k)\neq (0,0)\,.
$$

Since
\begin{align*}
& T_{\eta,\de} =  \partial_x^2 - \big(\eta + (x-\tfrac{\pi}{2}) \eta^2 + O(\eta^2 |\de| + |\eta|^3) \big)\,\partial_x \\
& \ \, - \big(3-3(\pi-2x)\eta -3\pi\eta\de + \frac34(\pi^2-12\pi x + 12x^2) \eta^2 + O(\eta^2|\de| + |\eta|^3) \big)\,,
\end{align*}
arguing as in Lemma \ref{L.1stOrderLambdaMu}, the eigenvalue equation reads as
\begin{align*}
0 & = T_{\eta,\de} \Phi_0 + \Lambda_0(\eta,\de) \Phi_0 \\
& = L \Phi_{00} + \de \Big( L \Phi_{01} + \Lambda_{01} \Phi_{00} \Big) + \de^2 \Big( L \Phi_{02} + \Lambda_{02} \Phi_{00} + \Lambda_{01} \Phi_{01} \Big) \\
& \quad + \eta \Big( L \Phi_{10} - \partial_x \Phi_{00} + \big( \Lambda_{10} +3(\pi -2x) \big)\Phi_{00} \Big) \\
& \quad + \eta^2 \Big( L \Phi_{20} - \partial_x \Phi_{10} + (\tfrac{\pi}{2}-x) \partial_x \Phi_{00} \\
& \hspace{1.5cm}  + \big(\Lambda_{10}+3(\pi-2x)\big) \Phi_{10} + \big(\Lambda_{20} - \tfrac34(\pi^2-12\pi x + 12x^2) \big) \Phi_{00} \Big) \\
& \quad + \eta \de \Big( L \Phi_{11}+ \big(\Lambda_{10} + 3(\pi - 2x) \big) \Phi_{01} \\
&  \hspace{1.5cm}   - \partial_x \Phi_{01}  + \Lambda_{01} \Phi_{10} + (\Lambda_{11}+3\pi) \Phi_{00} \Big) + O\big(|\eta|^3+|\de|^3\big) \quad \textup{in } (0,\pi)\,,
\end{align*}
with $L:= \partial_x^2 + 1$. By the choice of $\Lambda_{00}$ and $\Phi_{00}$, the zeroth order term of the above expression cancels. Moreover, arguing as in Lemma \ref{L.1stOrderLambdaMu}, it is immediate to see that $\Lambda_{01} = \Lambda_{02} = 0$ and that $\Phi_{01} \equiv \Phi_{02} \equiv 0$. 

The first nontrivial equation we need to solve is
\begin{equation*}  
L \Phi_{10} - \partial_x \Phi_{00} + \big( \Lambda_{10} +3(\pi -2x) \big)\Phi_{00} = 0 \quad \textup{in } (0,\pi)\,.
\end{equation*}
Multiplying by $\Phi_{00}$ and integrating by parts, we get 
$$
\Lambda_{10}\, \frac{\pi}{2} = \Lambda_{10} \int_0^\pi \sin(x)^2 \,dx = 3 \int_0^\pi (2x-\pi)\sin(x)^2 \,dx = 0\,,
$$
so $\Lambda_{10} = 0$. The equation can then be rewritten as
\begin{equation} \label{E.2nd1}
L \Phi_{10} - \cos(x) + 3(\pi-2x) \sin(x) = 0 \quad \textup{ in } (0,\pi)\,.
\end{equation}
Since $\Phi_{10}(0)=\Phi_{10}(\pi)=0$, the solution can be obtained using Fourier series, and a tedious but elementary computation shows
\begin{equation} \label{E.Phi10Serie}
\Phi_{10}(x) = \frac4{\pi}\sum_{k =1}^\infty \frac{2k(4k^2-13)}{(1-4k^2)^3} \sin(2kx)\,.
\end{equation}
The next equations we have are
\begin{align}
& L \Phi_{11} + (\Lambda_{11}+3\pi) \Phi_{00} = 0 \quad \textup{in } (0,\pi)\,, \label{E.2nd2} \\
& L \Phi_{20} + (\tfrac{\pi}{2}-x) \partial_x \Phi_{00} + 3(\pi-2x)  \Phi_{10} \label{E.2nd3}  \\ \nonumber
& \hspace{0.85cm} - \partial_x \Phi_{10} + \big(\Lambda_{20} - \tfrac34(\pi^2-12\pi x + 12x^2)\big) \Phi_{00} = 0 \quad \textup{in }(0,\pi)\,. 
\end{align}
Multiplying~\eqref{E.2nd2} by $\Phi_{00}$ and integrating by parts, we readily get $\Lambda_{11} = -3\pi$ and $\Phi_{11} \equiv 0$. Finally, doing the same with~\eqref{E.2nd3}, we obtain
$$
\begin{aligned}
\Lambda_{20} = \frac2{\pi} \int_0^\pi \Big(  & \partial_x \Phi_{10} - (\tfrac{\pi}{2}-x) \cos(x)  \\
& - 3(\pi-2x)  \Phi_{10} +  \tfrac34(\pi^2-12\pi x + 12x^2) \sin(x) \Big) \sin(x) \, dx \,.
\end{aligned}
$$
Plugging in the expression for $\Phi_{10}$ that we found in~\eqref{E.Phi10Serie}, a straightforward computation yields 
$
\Lambda_{20} = -16
$, so
 we arrive at the asymptotic formula
$$
\Lambda_0(\eta,\de) = 4-3\pi\eta\de -16\eta^2 + O(|\eta|^3+|\de|^3)\,.
$$

We now compute $\nuT_2$. We then start with $\nuT_{00} := 4$ and $\psiT_{00}(x): = \cos(2x)$, which leads to the formulas

\begin{align*}
	\widetilde{\Psi}_2(\eta,\de)(x) &= \sum_{j,k\geq 0} \eta^j \de^k\,  \psiT_{jk}(x)\,, \quad \textup{with} \quad  \psiT_{jk}'(0) = \psiT_{jk}'(\pi) = 0 \,, \quad \forall \ j,k \geq 0\,,\\[0.2cm]
	\widetilde{\nu}_2(\eta,\de) &= \sum_{j,k \geq 0} \eta^j \de^k\, \nuT_{jk}\,,
\end{align*}
with the normalization
$$
\int_{0}^{\pi} \psiT_2(\eta,\de)(x) \cos(2x)\, dx = \frac{\pi}{2}= \int_{0}^{\pi} \psiT_{00}(x) \cos(2x)\, dx \,.
$$
Taking into account the expansion of $T_{\eta,\de}$, it is clear that
$$
\widetilde{T}_{\eta,\de} =  \partial_x^2 - \big(\eta + (x-\tfrac{\pi}{2}) \eta^2 + O(\eta^2 |\de| + |\eta|^3) \big)\,\partial_x\,. \\
$$
Setting $\widetilde{L}:= \partial_x^2 + 4$, we then get
\begin{align*}
0 & = \widetilde{T}_{\eta,\de} \psiT_2 + \nuT_2(\eta,\de) \psiT_2 \\
& = \widetilde{L} \psiT_{00} + \de \Big( \widetilde{L} \psiT_{01} + \nuT_{01} \psiT_{00} \Big) + \eta \Big( \widetilde{L} \psiT_{10} - \partial_x \psiT_{00} + \nuT_{10}\psiT_{00} \Big) \\
& \quad + \eta^2 \Big( \widetilde{L} \psiT_{20} - \partial_x \psiT_{10} + (\tfrac{\pi}{2}-x) \partial_x \psiT_{00} + \nuT_{10} \psiT_{10} + \nuT_{20} \psiT_{00} \Big) \\
& \quad + \eta \de \Big( \widetilde{L} \psiT_{11} - \partial_x \psiT_{01} + \nuT_{10}  \psiT_{01} + \nuT_{01} \psiT_{10} + \nuT_{11} \psiT_{00} \Big) \\
& \quad + \de^2 \Big( \widetilde{L} \psiT_{02} + \nuT_{02} \psiT_{00} + \nuT_{01} \psiT_{01} \Big) + O\big(|\eta|^3+|\de|^3\big)\,, \quad \textup{in } (0,\pi)\,.
\end{align*}

Arguing as in the first case, it is not difficult to see that $\nuT_{01} = \nuT_{02} = \nuT_{10}$ and that $\psiT_{01} \equiv \psiT_{02} \equiv 0$. The equation for $\psiT_{10}$ therefore reads as
\begin{equation}
\widetilde{L} \psiT_{10} + 2\sin(2x) = 0 \quad \textup{in }(0,\pi)\,,
\end{equation}
so once again we can solve this with zero Neumann boundary conditions by means of Fourier series:
$$
\psiT_{10}(x) = -\frac{16}{\pi} \sum_{j = 0}^\infty \frac{1}{((2j+1)^2-4)^2} \cos((2j+1)x)\,.
$$

On the other hand, we also need that
\begin{align} \label{E.2nd4}
& \widetilde{L} \psiT_{11} + \nuT_{11} \psiT_{00} = 0 \quad \textup{in } (0,\pi)\,,\\ \label{E.2nd5}
& \widetilde{L} \psiT_{20} - \partial_x \psiT_{10} + (\tfrac{\pi}{2}-x) \partial_x \psiT_{00} + \nuT_{20} \psiT_{00} = 0 \quad \textup{in } (0,\pi)\,.
\end{align}
Using $\psiT_{00}$ as test function in \eqref{E.2nd4} and integrating by parts, we get $\nuT_{11} = 0$ and $\psiT_{11} \equiv 0$. Doing the same in \eqref{E.2nd5}, it follows that
$$
\nuT_{20} = \frac2{\pi} \int_0^{\pi} \big( \partial_x \psiT_{10} + (\pi-2x)\sin(2x) \big) \cos(2x)\, dx\,.
$$
Direct computations using the explicit series expansion of $\psiT_{10}$ give $\nuT_{20} = \tfrac34$. Thus, we conclude that$$
\nuT_2(\eta,\de) = 4 + \frac34 \eta^2 + O(|\eta|^3+|\de|^3)\,.
$$
The lemma then follows.
\end{proof}

\subsubsection*{Step 4: The transversality condition}

Armed with the above estimates, we can now finish the proof of Proposition \ref{P.cross2}. We argue as in Lemma \ref{L.epsilonl}. That is, given the parameters $l\gg1$ and $h=(1-a)/\pi\ll1$, we consider the small constants $(\eta,\de)$ for which
\begin{equation}\label{E.fixed}
h= \eta[1-\tfrac\pi 2\eta(1+\de)]\,, %\qquad hl = 3[1-\tfrac\pi 2\eta(1+\de)]
\end{equation}
and recall that in this case
\[
 \nuT(h):= h^2 \MU(a) =\nuT_2(\eta,\de)\,,\qquad \la(h):= h^2\lambda_{l,0}(a)=\Lambda_0(\eta,\de)\,.
\]
By~\eqref{E.fixed}, the derivative of $\la$ with respect to the parameter~$h$, leaving~$l$ fixed, is
\[
\la'(h)=-\frac2{\pi \eta^2}\frac\pd{\pd \de}\La_0(\eta,\de)= \frac2{\pi \eta^2}[3\pi\eta + O(\de^2+\eta^2)] \,,
\]
where we have used the asymptotics of Lemma~\ref{L.2ndOrder}. Note that this asymptotic expansion can be differentiated term by term because it defines an analytic function of~$(\eta,\de)$.

Similarly,
\[
\nuT'(h)=-\frac2{\pi \eta^2}\frac\pd{\pd \de}\nuT_2(\eta,\de)= \frac2{\pi \eta^2}O(\de^2+\eta^2)\,.
\]
Since $\eta_l=\sqrt3/l$ and $\de_l=O(l^{-1})$, we then get that
\[
\nuT'(h_l)=\frac2{\pi \eta_l^2}O(l^{-2})\neq \frac2{\pi \eta^2}[3\sqrt3\pi l^{-1} + O(l^{-2})]=\la'(h_l)
\]
for large~$l$, so~\eqref{T} follows.
\end{proof}

\section{Setting up the problem}
\label{S.deform}

Given a constant $a\in(0,1)$ and functions $b,B\in C^{2,\al}(\TT)$ satisfying
\[
\|b\|_{L^\infty(\TT)}+\|B\|_{L^\infty(\TT)}<\min\Big\{a,\tfrac{1-a}{2}\Big\}\,,
\]
we consider the bounded domain defined in polar coordinates by
\begin{equation}\label{E.defOm}
\Om_a^{b,B}:=\{(r,\te)\in\RR^+\times\TT: a+b(\te)< r< 1+B(\te)\}\,.
\end{equation}

We aim to show that there exist some constant $a>0$ and some nonconstant functions $b$ and $B$ such that there is a Neumann eigenfunction~$u$ on the domain $\Om_a^{b,B}$, with eigenvalue $\MU(a)$, which is locally constant on the boundary. Equivalently, the function $u$ satisfies the equation
\begin{equation}\label{E.uOm}
	\De u+\MU(a) u=0\quad\text{in }\Om_a^{b,B}\,,\qquad \nabla u=0\quad\text{on }\pd\Om_a^{b,B}\,.
\end{equation}
Since finding the functions~$b$ and~$B$ is the key part of the problem, we shall start by fixing some $a>0$ and by mapping the fixed annulus $\Om_\frac12:=\big\{(R,\te)\in \big(\tfrac12,1\big)\times\TT\big\}$
into $\Om_a^{b,B}$
through the $C^{2,\al}$ diffeomorphism 
\begin{equation}\label{E.defPhi}
\Phi_a^{b,B}:\Om_\frac12\to \Om_a^{b,B}
\end{equation}
defined in polar coordinates as $\Om_\frac12\ni(R,\te)\mapsto (r,\te)\in \Om_a^{b,B}$, with
\[
r:=a+(1-a +B(\te))(2R-1)+2(1-R)b(\te) \,.
\]
For later purposes, we also introduce the shorthand notation 
\[
\Phi_a := \Phi_a^{0,0}\,,
\]
and denote the nontrivial component of this diffeomorphism by $\Phi_a^1$, which only depends on the radial variable on~$\Om_\frac12$:
$$
\Phi_a(R,\te) = (\Phi_a^1(R),\te) \,.
$$
Denoting by $\psi_{0,2}$, $\vp_{l,0}$ the radial parts of the corresponding Neumann and Dirichlet eigenfunctions of the annulus $\Om_a$, as in Section~\ref{S.eigen}, we also find it convenient to set
\begin{equation} \label{E.eigenvaluel.fixed.annulus}
 \overline{\psi}_a:= \psi_{0,2} \circ \Phi_a^1 \in C^{2,\alpha}([\tfrac12,1]) \quad \textup{and} \quad 
\overline{\varphi}_a := \varphi_{l,0} \circ \Phi_a^{1} \in C^{2,\alpha}([\tfrac12,1])\,,
\end{equation}
which are functions defined on the interval $[\frac12,1]$.

Using this change of variables, one can rewrite Equation~\eqref{E.uOm} in terms of the function
\[
v:=u\circ\Phi_a^{b,B}\in C^{2,\al}(\Om_\frac12)\,,
\]
as
\begin{equation} \label{E.problem.fixed.annulus}
L_a^{b,B}v=0\quad \text{in }\Om_\frac12\,,\qquad \nabla v=0\quad \text{on }\pd\Om_\frac12\,,
\end{equation}
where the differential operator 
\[
L_a^{b,B}v:= [\De (v\circ (\Phi_a^{b,B})^{-1})]\circ\Phi_a^{b,B} +\MU(a)v
\]
is simply $\De+\MU(a)$ written in the coordinates $(R,\te)$. A tedious but straightforward computation (see Appendix \ref{A.PullBack} for the details) yields

\begin{align*} % \label{P.L}
& L_a^{b,B}v = \frac{1}{4(1-a+B(\theta)-b(\theta))^2}\, \partial^2_R v \\
& \ + \frac{1}{2(1-a+B(\theta)-b(\theta))(a+(1-a +B(\te))(2R-1)+2(1-R)b(\te))} \partial_R v \\
	& \  + \frac{1}{(a+(1-a +B(\te))(2R-1)+2(1-R)b(\te))^2} \Bigg[ \partial_\theta^2 v \\
	&\hspace{3.9cm} + \left(\frac{B'(\theta)(2R-1) + 2(1-R)b'(\theta)}{2(1-a+B(\theta)-b(\theta))} \right)^2 \partial_R^2 v \\
	& \hspace{3.9cm} + \bigg( \frac{(B'(\theta)-b'(\theta))^2(2R-1)+b'(\theta)(B'(\theta)-b'(\theta))}{(1-a+B(\theta)- b(\theta))^2} \\ 
	&\hspace{3.9cm} - \frac{B''(\theta)(2R-1) + 2(1-R)b''(\theta)}{2(1-a+B(\theta)-b(\theta))} \bigg) \partial_R v \\
	& \hspace{3.9cm}  - \frac{B'(\theta)(2R-1) + 2(1-R)b'(\theta)}{1-a+B(\theta)-b(\theta)}\, \partial_{R\theta} v \Bigg] + \mu_{0,2}(a) v\,.
\end{align*}
 
\noindent In particular, 
\begin{equation} \label{E.La}
\begin{aligned}
L_a v := L_a^{0,0} v  = \frac{1}{4(1-a)^2} \bigg[\,  \partial_R^2 v & + \frac{1}{R-\frac12+\frac{a}{2(1-a)}} \partial_R v \\
&  + \frac{1}{(R-\frac12+\frac{a}{2(1-a)})^2} \partial_\theta^2 v \bigg] + \mu_{0,2}(a)v \,.
\end{aligned}
\end{equation}

Let us now present the functional setting that we will use to analyze these operators.
In what follows we will always consider function spaces with dihedral symmetry, that is, spaces of functions which are invariant under the action of the isometry group of a  $l$-sided regular polygon, for $l \geq 3$. More precisely, let us define
$$
C\ev^{k,\alpha}(\overline{\Omega}_\frac12):= \big\{u \in C^{k,\alpha}(\overline{\Omega}_\frac12): u(R,\te)=u(R,-\te)\,, \  u(R,\te) = u (R,\te+\tfrac{2\pi}l)\big\}\,,
$$
whose canonical norm we shall denote by
$$
\|u\|_{C^{k,\alpha}} := \|u\|_{C^{k,\alpha}(\overline{\Omega}_\frac12)}\,.
$$

Following the recent paper \cite{Fall-Minlend-Weth-main}, let us define the ``anisotropic'' space
$$
\mathcal{X}^{k,\al}:= \big\{ u \in C^{k,\alpha}\ev (\overline{\Omega}_\frac12): \partial_R u \in C^{k,\alpha}\ev(\overline{\Omega}_\frac12) \big\}\,,
$$
endowed with the norm
$
\|u\|_{\mathcal{X}^{k,\al}} := \|u\|_{C^{k,\alpha}} + \|\partial_R u \|_{C^{k,\alpha}}
$,
and its closed subspaces
$$
\mathcal{X}^{k,\al}\D := \big\{u \in \mathcal{X}^{k,\alpha}: u = 0 \textup{ on } \partial \Omega_{\frac12} \big\}$$
and
$$
\mathcal{X}^{k,\alpha}\DN  := \big\{u \in \mathcal{X}^{k,\alpha}: u = \partial_R u =0 \textup{ on } \partial \Omega_{\frac12} \big\}\,.
$$
For convenience, we also introduce the shorthand notation
$$
\mathcal{X}:= \mathcal{X}\D^{2,\alpha} \quad \textup{ and } \quad \|u\|_{\cX} := \|u\|_{\cX^{2,\alpha}}\,.
$$
We  also need the space
$$
\mathcal{Y}:= C\ev ^{1,\alpha}(\overline{\Omega}_\frac12) + \mathcal{X}\D ^{0,\alpha}\,,
$$
endowed with the norm
$$
\|u\|_{\mathcal{Y}} := \inf\big\{ \|u_1\|_{C^{1,\alpha}}+\|u_2\|_{\cX^{0,\alpha}} : u_1 \in C\ev ^{1,\alpha}(\overline{\Omega}_\frac12),\ u_2 \in \mathcal{X}\D ^{0,\alpha},\ u = u_1+u_2 \big\}\,.
$$
As discussed in \cite[Remark 3.2]{Fall-Minlend-Weth-main}, it is not difficult to see that $(\mathcal{Y},\|\cdot \|_{\mathcal{Y}})$ is a Banach space.

%Given an integer $l\geq3$, we shall also consider the subspaces
%\[
%\cX := \big\{u \in \cX : u(R,\te) = u (R,\te+\tfrac{2\pi}l) \big\} \quad \textup{ and } \quad \cY := \big\{u \in \cY : u(R,\te) = u (R,\te+\tfrac{2\pi}l) \big\}
%\]
%of functions that are invariant under the rotations of angle $2\pi/l$ (that is, under the dihedral group~$\bD_l$, which is the symmetry group of an $n$-sided regular polygon). 
Furthermore, we define the open subset
\begin{equation*}
\mathcal{U}:= \Big\{(b,B) \in C^{2,\alpha}\ev(\TT) \times C^{2,\alpha}\ev(\TT) : \|b\|_{L^{\infty}(\TT)} + \|B\|_{L^{\infty}(\TT)} < \min\big\{a,\tfrac{1-a}{2}\big\} \Big\}\,.
\end{equation*}

A first result we will need is the following:

\begin{lemma} The following assertions hold true:
	\begin{enumerate}
		\item The function 
		$
		(v,b,B) \mapsto L_a^{b,B}\,[\,\overline{\psi}_a + v\,]
		$, 
with $\overline\psi_a$ given by~\eqref{E.eigenvaluel.fixed.annulus}, maps $\cX\DN ^{2,\alpha} \times \mathcal{U}$ into $\cY$.	
		\item If $v \in \cX$, then $L_a v \in \cY$.
	\end{enumerate}
	
\end{lemma}

\begin{proof}
	Let $(w,b,B) \in \cX^{2,\alpha} \times \mathcal{U}$ be fixed but arbitrary. We define
	\begin{align*}
		& F_1(w,b,B) := \frac{1}{4(1-a+B(\theta)-b(\theta))^2}\, \partial^2_R w \\
		& \  + \frac{1}{2(1-a+B(\theta)-b(\theta))(a+(1-a +B(\te))(2R-1)+2(1-R)b(\te))} \partial_R w \\
		& \  + \frac{1}{(a+(1-a +B(\te))(2R-1)+2(1-R)b(\te))^2}\, \times\, \\[0.3cm] & \hspace{0.5cm}  \Bigg[  \left(\frac{B'(\theta)(2R-1) + 2(1-R)b'(\theta)}{2(1-a+B(\theta)-b(\theta))} \right)^2 \partial_R^2 w \\
		& \hspace{3.8cm}  - \frac{B'(\theta)(2R-1) + 2(1-R)b'(\theta)}{1-a+B(\theta)-b(\theta)}\, \partial_{R\theta} w \Bigg] + \MU(a)  w \,,
	\end{align*}
	and
	\begin{align*}
		& F_2(w,b,B) :=  \frac{1}{(a+(1-a +B(\te))(2R-1)+2(1-R)b(\te))^2} \Bigg[ \partial_\theta^2 w  \\
		& \ + \bigg( \frac{(B'(\theta)-b'(\theta))^2(2R-1)+b'(\theta)(B'(\theta)-b'(\theta))}{(1-a+B(\theta)- b(\theta))^2} \\
		& \ - \frac{B''(\theta)(2R-1) + 2(1-R)b''(\theta)}{2(1-a+B(\theta)-b(\theta))} \bigg) \partial_R w \Bigg]\,, \\
	\end{align*}
and observe that 
$$
L_a^{b,B}w = F_1(w,b,B) + F_2(w,b,B)\,.
$$
Then, it is easy to check that $F_1(w,b,B) \in C\ev ^{1,\alpha}(\overline{\Omega}_\frac12)$ and that $F_2(w,b,B) \in \mathcal{X}\D ^{0,\alpha}$ for $w = \overline{\psi}_a + v$ with $v \in \cX\DN ^{2,\alpha}$. In this way we conclude (i). Analogously, we can prove (ii), using that in this case $b\equiv B\equiv 0$.
\end{proof}

Our goal now is to find some constant $a \in (0,1)$ and some functions $v \in C^{2,\alpha}(\overline{\Omega}_\frac12)$ and $b,B \in C^{2,\alpha}(\TT)$ such that \eqref{E.problem.fixed.annulus} admits a non-trivial solution. Using the functional setting that we just presented, we can reduce our problem so that the basic unknowns are just a constant $a \in (0,1)$ and a function $v \in \cX$. 

To this end, we define the map
$$
\cX \to \cX\DN ^{2,\alpha} \times C\ev^{2,\alpha}(\TT) \times C\ev^{2,\alpha}(\TT), \qquad v \mapsto ( w _v, b_v, B_v)\,,
$$
where
\begin{equation}\label{E.bvBvuv}
\begin{aligned}
  b_v(\te)&:= -2(1-a)(\overline{\psi}_a''(\tfrac12))^{-1} \partial_R v\big(\tfrac12,\theta\big)\,,  \\
  B_v(\te) &:= -2(1-a)(\overline{\psi}_a''(1) )^{-1} \partial_R v(1,\te)\,,  \\
  w _v(R,\te) &:= v(R,\te) + \frac{\overline{\psi}_a'(R)}{2(1-a)}\Big[2(1-R)b_v(\te) + (2R-1) B_v(\te) \Big]\,, 
 \end{aligned}
 \end{equation}
%$$
%b_v(\te):= \big(\overline{\psi}_a''(\tfrac12)\big)^{-1} \partial_R v\big(\tfrac12,\theta\big) \quad \textup{ and } \quad B_v(\te) := ( \overline{\psi}_a''(1) )^{-1} \partial_R v(1,\te)\,, \quad \textup{ for } \te \in \TT\,,
%$$
%and
%$$ 
%u_v(R,\te) := v(R,\te) - \overline{\psi}_a'(R)\big[2(1-R)b_v(\te) + (2R-1) B_v(\te) \big]\,, \quad \textup{ for } (R,\te) \in \Omega_\frac12\,,
%$$
and introduce the open subset
\begin{align*}
	\mathcal{O}&:= \Big\{v \in \cX : \|b_v\|_{L^{\infty}(\TT)} + \|B_v\|_{L^{\infty}(\TT)} < \min\big\{a,\frac{1-a}{2}\big\} \Big\}.
%\mathcal{O}_l&:=  \big\{v \in \cX : \|b_v\|_{L^{\infty}(\TT)} + \|B_v\|_{L^{\infty}(\TT)} < \min\big\{a,\tfrac{1-a}{2}\big\}  \Big\}
\end{align*}
%with $l\in\NN$.

Following~\cite{Fall-Minlend-Weth-main}, we now define the function
\begin{equation} \label{E.Ga}
G_a: \mathcal{O} \to \mathcal{Y}\,,  \qquad v \mapsto L_a^{b_v, B_v}\,[\,\overline{\psi}_{a} +  w _v\,]\,,
\end{equation}
%which obviously maps $\mathcal O_l\to\cY$.
If $v\in\mathcal X$ satisfies $G_a(v)=0$, where $a$ is any constant in $(0,1)$, then $\widetilde{ w }:= \overline{\psi}_a +  w _v$ is the desired solution to \eqref{E.problem.fixed.annulus} with $b := b_v$,  $B := B_v$. Hence, we can prove Theorem~\ref{T.main} by finding nontrivial zeros of the operator $G_a$ for some $a \in (0,1)$. This will be done in the following section.

\begin{remark} \label{R.ansatz}
The rationale behind the definitions~\eqref{E.bvBvuv} is the following. Recall that our goal to construct our solution by bifurcating from the second nontrivial radial Neumann eigenfunction of an annulus, namely $\psi_{0,2}(r)$. The choice of the functions $b_v$, $B_v$ and $w_v$ in terms of $v$ is precisely motivated by the first order expansion of 
$$
\overline{\psi}_{0,2}^{\,sb,\,sB} := \psi_{0,2} \circ \Phi_a^{\,sb,\,sB}\,,
$$
at $s = 0$, where we are writing $\psi_{0,2}(r,\te) \equiv \psi_{0,2}(r)$ with some abuse of notation. Observe that
$$
\psi_{0,2}^{0,0} = \psi_{0,2} \circ \Phi_a^1 = \overline{\psi}_a\,,
$$
so
$$
 \dds\,  \overline{\psi}_{0,2}^{\,sb, sB} \Big|_{s=0}   =  \big((2R-1) B(\te) + 2(1-R) b(\te) \big)  \big(\psi_{0,2}' \circ \Phi_a^1 \big) = w_a^{b, B}\,,
$$
with 
$$
w_{a}^{b, B}(R,\te) := \frac{\overline{\psi}_a'(R)}{2(1-a)}\Big[2(1-R)b(\te) + (2R-1) B(\te) \Big]\,.
$$
Hence, it is natural to look for a solution to \eqref{E.problem.fixed.annulus} of the form
$$
\widetilde{v} := \overline{\psi}_a + w_a^{b,B} + v\,,
$$
with $v$, $b$ and $B$ small. We have thus chosen $b_v$ and $B_v$ so that, for any $v \in \cX$, the function $w_v :=  w_a^{b_v,B_v} + v$ is in the space~$\cX_{\rm DN}^{2,\al}$. 
\end{remark}

% This will be done in the next section using the classical Crandall--Rabinowitz bifurcation theorem. 
   
%\begin{equation}
%\De \psi_{0,2}+\MU(a) \psi_{0,2}=0\quad\text{in }\Om_a\,,\qquad %\nabla \psi_{2,0}=0\quad\text{on}\pd\Om_a\,.
%\end{equation}

\section{The bifurcation argument}
\label{S.bifurcation}

In this section we present the proof of Theorem \ref{T.main}. The basic idea is to apply the classical Crandall--Rabinowitz bifurcation theorem (see e.g. \cite{M.CR} or \cite[Theorem I.5.1]{Kielhofer}) to the operator $G_a$ introduced in the previous section. The key ingredients to make the argument work are the asymptotic estimates for the eigenvalues of an annulus proved in Section~\ref{S.eigen}.

The first auxiliary result we need is the following:

\begin{lemma}\label{L.DGa}
For all $a \in (0,1)$, the map $G_a: \mathcal{O} \to \mathcal{Y}$ is smooth. Moreover,
$$
DG_a(0)v = L_a v\quad  \textup{ for all } v \in \mathcal{X}\,,
$$
where $L_a$ is the operator defined in~\eqref{E.La}.
\end{lemma}

\begin{proof}
The smoothness of $G_a$ follows immediately from its definition, so we just have to show that $DG_a(0) = L_a$. As $b_v,B_v,w_v$ depend linearly on~$v$ and
\[
w_v= v + w_{a}^{b_v, B_v}
\]
with
$$
w_{a}^{b_v, B_v}(R,\te) := \frac{\overline{\psi}_a'(R)}{2(1-a)}\Big[2(1-R)b_v(\te) + (2R-1) B_v(\te) \Big]\,,
$$
it is easy to see that
\begin{equation}\label{E.DGa02}
\begin{aligned}
DG_a(0)v & = \dds L_a^{sb_v, sB_v}[\overline{\psi}_a+sw_v] \Big|_{s=0}\\ & = L_a[v+w_a^{b_v,B_v}] + \dds L_a^{sb_v, s B_v} \Big|_{s=0} [\, \overline{\psi}_a ]\,.
\end{aligned}
\end{equation}

To compute the second term, note that the function $\psi_{0,2}(r)$ (see \eqref{E.ODE.Neumann} in Section \ref{S.eigen}) satisfies the ODE
$$
\psi_{0,2}''+ \frac{\psi_{0,2}'}{r} + \MU(a)\psi_{0,2} = 0
$$
for all $r>0$. Therefore, writing as in Remark \ref{R.ansatz} $\psi_{0,2}(r,\te)\equiv \psi_{0,2}(r)$ by abuse of notation, and defining 
$
\overline{\psi}_{0,2}^{\,b,B} := \psi_{0,2} \circ \Phi_a^{b,B},
$
we trivially have
$$
L_a^{b,B}\, \overline{\psi}_{0,2}^{\,b,B} = 0\,, \quad \textup{ in } \Omega_{\frac12}\,.
$$
Substituting $(b,B):=(sb_v,sB_v)$ and differentiating the resulting identity, we then find
\begin{equation}\label{E.DGa01}
\begin{aligned}
0 & = \dds \left( L_a^{s b_v, sB_v}\, \overline{\psi}_{0,2}^{\, sb_v, sB_v} \right)\bigg|_{s = 0} \\
&= \dds L_a^{sb_v, sB_v} \Big|_{s=0} [\, \overline{\psi}_a ] + L_a \left[ \dds \overline{\psi}_{0,2}^{\,sb_v, sB_v} \Big|_{s=0} \right] \\
& = \dds L_a^{sb_v, sB_v} \Big|_{s=0} [\, \overline{\psi}_a ] + L_a w_a^{b_v,B_v}\,.
\end{aligned}
\end{equation}
Since $L_a$ is a linear operator, the result immediately follows by combining \eqref{E.DGa02} and \eqref{E.DGa01}.
\end{proof}

Next we need a regularity result for the operator
$$
\widetilde{L}_a v:= \frac{1}{4(1-a)^2} \bigg[ \partial_R^2 v + \frac{1}{R-\frac12+\frac{a}{2(1-a)}} \partial_R v + \frac{1}{(R-\frac12+\frac{a}{2(1-a)})^2} \partial_\theta^2 v \bigg]\,.
$$

\begin{lemma}\label{L.regularity}
For all $a \in (0,1)$, let $f \in C\ev ^{0,\alpha}(\overline{\Omega}_\frac12)$ and $v \in C\ev ^{2,\alpha}(\overline{\Omega}_\frac12)$ satisfy
\begin{equation} \label{E.regularity}
\widetilde{L}_a v = f, \quad \textup{in }\Om_\frac12\,,\qquad v=0\quad \textup{on }\pd\Om_\frac12\,.
\end{equation}
If $f \in \cY$, then $v \in \cX$. 
\end{lemma}

\begin{proof}
Taking into account the definition of $\cX$, we only need to prove that $f \in \cY$ implies $\pd_R v \in C\ev^{2,\alpha}(\overline{\Omega}_\frac12)$. Setting $w:= \pd_R v \in C\ev ^{1,\alpha}(\overline{\Omega}_\frac12)$ and differentiating \eqref{E.regularity}, we find that $w$ satisfies
\begin{equation} \label{E.regularity.w}
\begin{aligned}
\widetilde{L}_a w & = \pd_R f \\
& \quad + \frac{1}{4(1-a)^2} \bigg[ \frac{1}{(R-\frac12+\frac{a}{2(1-a)})^2}\, w + \frac{2}{(R-\frac12+\frac{a}{2(1-a)})^3}\, \pd_\te^2 v \bigg], 
\end{aligned}
\end{equation}
in $\Om_\frac12$, in the distributional sense. Note that the right hand side is in $C^{0,\alpha}(\overline{\Omega}_\frac12)$ when $f\in\cY$.

Isolating $\pd_R^2 v|_{\pd\Om_{\frac12}}$ in Equation \eqref{E.regularity} and using that $v|_{\pd\Om_{\frac12}}=0$, we easily see that
\begin{equation} \label{E.regularity.w.boundary}
\pd_R w\, \big|_{\pd \Omega_\frac12} = 4(1-a)^2 \bigg( f - \frac{1}{R-\frac12+\frac{a}{2(1-a)}} \pd_R v \bigg)\bigg|_{\pd \Omega_\frac12}\,.
\end{equation}
Decomposing $\cY\ni f = f_1 + f_2 $ with $f_1 \in C\ev ^{1,\alpha}(\overline{\Omega}_\frac12)$ and $f_2 \in \cX\D ^{0,\alpha}$, we then see that $f|_{\pd \Omega_\frac12} = f_1 \in C^{1,\alpha}(\partial \Omega_\frac12)$ so $\pd_Rw|_{\Om_{\frac12}} \in C^{1,\alpha}(\partial\Omega_\frac12)$. Standard elliptic regularity theory applied to the Neumann problem \eqref{E.regularity.w}-\eqref{E.regularity.w.boundary} then shows that $\pd_R v = w \in C\ev^{2,\alpha}(\overline{\Omega}_\frac12)$, so the result follows. 
\end{proof}

In the following lemmas, we shall show that $G_a:\mathcal O\to \cY $ satisfies the assumptions of the Crandall--Rabinowitz theorem. The first result we need is the following:

\begin{lemma} \label{L.Fredholm}
For all $a \in (0,1)$ and all $l \in \NN$, $L_a: \cX \to \cY$ is a Fredholm operator of index zero. 
\end{lemma}
 
\begin{proof}
Since the property of being a Fredholm operator and the Fredholm index of an operator are preserved by compact perturbations, and since the embedding $\cX \hookrightarrow \cY$ is compact, it suffices to show that $\widetilde{L}_a= L_a-\MU(a) $ defines a Fredholm operator $\cX \to \cY$ of index zero. To prove this, we shall prove the stronger statement that $\widetilde{L}_a$ is a topological isomorphism. Since $\widetilde{L}_a$ is a linear continuous map, by the open mapping theorem, it suffices to show that $\widetilde{L}_a: \cX \to \cY$ is a bijective map. 

We first prove the injectivity. Let $v \in \cX$ such that $\widetilde{L}_a v = 0$ in $\Omega_\frac12$. Defining $w:= v \circ \Phi_a^{-1}$, we get that
$$
\Delta w = 0 \quad \textup{in } \Omega_a\,, \qquad w = 0 \quad \textup{on } \pd \Omega_a\,.
$$ 
Hence, it follows that $w\equiv 0$. This implies that $v\equiv0$ and the injectivity of $\widetilde{L}_a$ follows.

We now prove that $\widetilde{L}_a$ is onto. In other words, given $f \in \cY$, we show that there exists $v \in \cX$ solving
\begin{equation}\label{E.surjectivity}
	\widetilde{L}_a v = f \quad \textup{in } \Omega_\frac12\,, \qquad v = 0 \quad \textup{on } \pd \Omega_\frac12\,.
\end{equation}
Observe that $ f \in C^{0,\alpha}_l(\overline{\Om}_{\frac 1 2})$. Equation \eqref{E.surjectivity} is equivalent to:
\begin{equation}\label{E.surjectivity2}
	\Delta w = g \quad \textup{in } \Omega_a\,, \qquad w = 0 \quad \textup{on } \pd \Omega_a\,.
\end{equation}
where $w:= v \circ \Phi_a^{-1}$ and $g: =f \circ \Phi_a^{-1}$. The existence of a unique $C^{2,\alpha}$ solution to \eqref{E.surjectivity2} is known. Hence, we obtain a solution $v \in C^{2,\alpha}_l(\overline{\Om}_{\frac 1 2})$ to \eqref{E.surjectivity}. By Lemma \ref{L.regularity}, it follows that $v \in \cX$ and so that $\widetilde{L}_a$ is onto.
\end{proof} 
 
 We are now ready to present the bifurcation argument:
 
\begin{proposition} \label{P.CR}
	Let $l \geq 4$ such that \eqref{NR} and \eqref{T} are satisfied (see Propositions \ref{P.cross1}, \ref{P.cross2}). Take $a_l \in (0,1)$ as in Proposition \ref{P.cross1} and let $\overline{\varphi}_{a_l} \in C^{2,\alpha}([\tfrac12,1])$ be as in \eqref{E.eigenvaluel.fixed.annulus}.
	Then:
	\begin{enumerate}
		\item The kernel of $L_{a_l}$ is one-dimensional. Furthermore, \begin{equation} \label{E.Kernel}
			\Ker(L_{a_l}) = {\rm span}\{\overline{v}\}, \quad \textup{ with } \quad \,\overline{v}(R,\theta):= \overline{\varphi}_{a_l}(R) \cos(l\te)\,.
		\end{equation}
		\item The image of $L_{a_l}$ is given by
		$$
		{\rm{Im}}(L_{a_l}) = \left\{ w \in \Yl: \int_{\Omega_{\frac12}} \overline{v}(R,\theta) w(R,\theta) \Big(R - \frac12 +\frac{a_l}{2(1-a_l)}\Big) dR\, d\te = 0 \right\}\,.
		$$
		\item $L_{a_l}$ satisfies the transversality property, that is 
		$$
		\frac{{\rm d}}{{\rm d}a} L_a \Big|_{a = a_l} [\,\overline{v}\,]  \not\in {\rm Im}(L_{a_l})\,.
		$$
	\end{enumerate}
\end{proposition}

\begin{proof}
	(i) Let $w \in \Ker(L_{a_l})$ and define $z:= w \circ \Phi_{a_l}^{-1}$. Then, $w$ is a $l$-symmetric solution to the problem
	$$ \Delta z + \mu(a_l) z  =0 \quad \textup{in } \Omega_a\,, \qquad z = 0 \quad \textup{on } \pd \Omega_a\,.$$
By the analysis made in Section 2, we know that $z(r,\te) = t \varphi_{l,0}(r) \cos (l \te)$, with $t \in \RR$. Hence, (i) follows.

(ii) By Lemma \ref{L.Fredholm} we know that $L_{a_l}$ is a Fredholm operator of index zero. Thus, it is enough to prove that
$$
{\rm{Im}}(L_{a_l}) \subset \left\{ w \in \Yl: \int_{\Omega_{\frac12}} \overline{v}(R,\theta) w(R,\theta) \Big(R - \frac12 +\frac{a_l}{2(1-a_l)}\Big)\, dR \,d\te = 0 \right\}\,.
$$
Let $w $ be in the image set of $\mathcal L$, that is, $w = L_{a_l}u$ for some $u \in \Xl$. Since $\overline{v} \in \Ker(L_{a_l})$, integrating by parts, we obtain that
\begin{align*}
& \int_{\Omega_{\frac12}} \overline{v}(R,\theta) w(R,\theta) \Big(R - \frac12 +\frac{a_l}{2(1-a_l)}\Big) dR\, d\te \\
& \quad = \int_{\Omega_{\frac12}} L_{a_l}\overline{v}(R,\theta) u(R,\theta) \Big(R - \frac12 +\frac{a_l}{2(1-a_l)}\Big) dR\, d\te = 0\,, 
\end{align*}
and the desired inclusion follows. 

(iii) For all $a \in (0,1)$, we set $\overline{v}_a(R,\theta):= \overline{\varphi}_a(R) \cos(l\te) \in \Xl$, with $\overline{\varphi}_a$ as in \eqref{E.eigenvaluel.fixed.annulus}, and $\overline{w} := \tfrac{{\rm d}}{{\rm d}a} \overline{v}_a|_{a = a_l} \in \Xl$. Then, taking into account \eqref{E.ODE.Dirichlet}, we see that
$$
L_a \overline{v}_a = (\mu_{0,2}(a) - \lambda_{l,0}(a)) \overline{v}_a\,, \quad \textup{ in } \Omega_\frac12\,.
$$
Moreover, differentiating this identity with respect to $a$ and evaluating at $a = a_l$, we get that
$$
\frac{{\rm d}}{{\rm d}a} L_a \Big|_{a = a_l} [\,\overline{v}\,] + L_{a_l}\overline{w} = (\mu_{0,2}'(a_l) - \lambda_{l,0}'(a_l))\overline{v}\,, \quad \textup{ in } \Omega_\frac12\,.
$$
By Proposition \ref{P.cross2}, we know that $\mu_{0,2}'(a_l) - \lambda_{l,0}'(a_l) \neq 0$ and by (i), we have that $\overline{v} \in \Ker(L_{a_l})$. Hence,  $
\tfrac{{\rm d}}{{\rm d}a} L_a \big|_{a = a_l}[\,\overline{v}\,] \not\in {\rm Im}(L_{a_l})
$,  so we conclude that $L_{a_l}$ satisfies the transversality property. 
\end{proof}

We are now ready to prove Theorem \ref{T.main}. Using the notation we have introduced in previous sections, here we provide a more precise statement of this result. Let us recall that the domains of the form $\Omega_a^{b,B}$ were introduced in \eqref{E.defOm}.

\begin{theorem} \label{T.mainDetails}
Let $l \geq 4$ be such that \eqref{NR} and \eqref{T} are satisfied. This condition holds, in particular, for $l=4$ and for all large enough~$l$. Given $a_l \in (0,1)$ as in Proposition \ref{P.cross1}, there exist some $s_l > 0$ and a continuously differentiable curve
$$
\big\{(a(s), b_s, B_s): s \in (-s_l ,s_l ),\ (a(0), b_0, B_0) = (a_l,0,0) \big\} \subset (0,1) \times (C_l^{2,\alpha}(\TT))^2  \,,
$$
with
\begin{align*}
	b_s(\te) &= -2s(1-a(s)) \, \frac{\overline{\varphi}'_{a_l}(\tfrac12)}{\overline{\psi}''_{a(s)}(\tfrac12)}\,  \cos(l\te) + o(s)\,, \\
	B_s(\te) &= -2s(1-a(s)) \, \frac{\overline{\varphi}'_{a_l}(1)}{\overline{\psi}''_{a(s)}(1)}\,  \cos(l\te) + o(s)\,,
\end{align*}
such that the overdetermined problem 
\begin{equation*}
\De u_s+\MU(a(s)) u_s=0\quad\text{in }\Om_{a(s)}^{b_s,B_s}\,,\qquad \nabla u_s=0\quad\text{on }\pd\Om_{a(s)}^{b_s,B_s}\,,
\end{equation*}
admits a nonconstant solution for every $s \in (-s_l ,s_l )$. Moreover, both the solution $u_s$ and the boundary of the domain are analytic. 
\end{theorem}

\begin{proof} 
First, note that the fact that the conditions \eqref{NR} and \eqref{T} hold for all large enough~$l$ was established in Propositions \ref{P.cross1} and \ref{P.cross2}, while the case $l=4$ follows from Propositions \ref{P.cross1} and~\ref{l=4}.

By Proposition \ref{P.CR} we know that the map
\begin{equation}
\big(0,1\big) \times \mathcal{O} \to \cY\,, \qquad (a,v) \mapsto G_a(v)\,,
\end{equation}
satisfies the hypotheses of the Crandall--Rabinowitz theorem. Therefore, there exits a nontrivial continuously differentiable curve through $(a_l,0)$,
$$
\big\{(a(s), v_s): s \in (-s_l ,s_l ),\ (a(0), v_0) = (a_l,0) \big\}\subset (0,1)\times \mathcal O\,,
$$
such that
$$
G_{a(s)}(v_s) = 0\,, \quad \textup{for} \quad s \in (-s_l ,s_l )\,.
$$
Moreover, for $\overline{v}$ as in \eqref{E.Kernel}, it follows that
\begin{equation} \label{E.expansionvs}
v_s = s \, \overline{v} + o(s) \quad \textup{in } \cX \quad \textup{as} \quad s \to 0\,.
\end{equation}

As we saw in \eqref{E.bvBvuv}-\eqref{E.Ga}, for all $s \in (-s_l ,s_l )$, the function 
$$
\widetilde{ w }_s = \overline{\psi}_{a(s)} +  w _{v_s}
$$
is then a nontrivial solution to \eqref{E.problem.fixed.annulus} with $b := b_{v_s}$,  $B := B_{v_s}$ and $a := a(s)$. More explicitly,  
$$
\widetilde{ w }_s(R,\te) = \overline{\psi}_{a(s)}(R) + v_s(R,\te) + \frac{\overline{\psi}_{a(s)}'(R)}{2(1-a(s))} \Big[2(1-R)b_{s}(\te) + (2R-1)B_{s}(\te) \Big]\,,
$$
with
\begin{equation} \label{E.bvsBvs}
\begin{aligned}
b_{s}(\te) & := b_{v_s}(\te) = -2(1-a(s))(\overline{\psi}_{a(s)}''(\tfrac12))^{-1} \partial_R v_s\big(\tfrac12,\theta\big)\,,\\
B_{s}(\te) & := B_{v_s}(\te) = -2(1-a(s))(\overline{\psi}_{a(s)}''(1) )^{-1} \partial_R v_s(1,\te)\,.
\end{aligned}
\end{equation}
Furthermore, combining \eqref{E.Kernel} with \eqref{E.expansionvs},   we get
\begin{align*}
	b_s(\te) &= -2s(1-a(s)) \, \frac{\overline{\varphi}'_{a_l}(\tfrac12)}{\overline{\psi}''_{a(s)}(\tfrac12)}\,  \cos(l\te) + o(s)\,,\\[1mm]
	B_s(\te) &= -2s(1-a(s)) \, \frac{\overline{\varphi}'_{a_l}(1)}{\overline{\psi}''_{a(s)}(1)}\,  \cos(l\te) + o(s)\,.
\end{align*}
This immediately gives the formula presented in the statement of Theorem~\ref{T.main}.

To change back to the original variables $(r,\te) \in  \Omega_{a(s)}^{b_s,B_s}$, we simply define 
$$
u_s := \widetilde{ w }_s \circ \big(\Phi_{a(s)}^{b_s,B_s}\big)^{-1} 
$$
and conclude that, for all $s \in (-s_l ,s_l )$, $u_s \in C^{2,\alpha}(\overline{\Omega}_{a(s)}^{\,b_{v_s}, B_{v_s}})$ solves 
\begin{equation*}
\De u_s+\MU(a(s)) u_s=0\quad\text{in }\Om_{a(s)}^{b_s,B_s}\,,\qquad \nabla u_s=0\quad\text{on }\pd\Om_{a(s)}^{b_s,B_s}\,.
\end{equation*}
The analyticity of both the domains $\Om_{a(s)}^{b_s,B_s}$ and the solutions $u_s$ for all $s \in (-s_l ,s_l )$ follows from  Kinderlehrer--Nirenberg~\cite{KinderlehrerNirenberg}.
\end{proof}

\section{Applications of the main theorem}
\label{S.corollaries}

In this section we present the proofs of the two applications of the main theorem that we discussed in the Introduction:

\begin{proof}[Proof of Corollary~\ref{C.Pompeiu}]
Let $\Om$ be a domain as in Theorem~\ref{T.main}. Therefore, the boundary $\pd\Om$ (which is analytic) consists of two connected components $\Ga_1,\Ga_2$, and there exist two nonzero constants $c_1,c_2$ and a nonconstant Neumann eigenfunction~$u$, satisfying
\[
\De u+\mu u=0
\]
in~$\Om$, such that $u|_{\Ga_j}=c_j$ with $j=1,2$. The complement $\RR^2\backslash\BOm$ has a bounded connected component, which we will henceforth call~$\Om'$, and an unbounded component, which is then $\RR^2\backslash\overline{(\Om\cup \Om')}$. Relabeling the curves~$\Ga_j$ is necessary, we can assume that $\pd\Om'=\Ga_2$.

Now let $f$ be any function such that 
\[
\De f + \mu f=0
\]
on~$\RR^2$, for instance $f(x)=\sin(\mu^{1/2} x_1)$, and set 
\begin{equation}\label{E.defc}
c:= \frac{c_2}{c_1}-1\,.
\end{equation}
As the local maxima of a Bessel function are decreasing, it is easy to see that $c>0$.

We claim that, for any rigid motion~$\cR$,
\begin{equation}\label{E.Pompclaim}
0=\int_{\cR(\Om)}f\, dx - c\int_{\cR(\Om')}f\, dx=\int_{\RR^2} f\circ\cR^{-1}\, \rho\, dx\,,
\end{equation}
where we have introduced the function 
\[
\rho(x):= \mathbbm 1_\Om(x) - c\,\mathbbm 1_{\Om'}(x)\,.
\]

Integrating by parts, one readily finds that the compactly supported $C^{1,1}$~function
\[
w(x):=\begin{cases}
	0 & \text{if } x\in\RR^2\backslash\overline{(\Om\cup\Om')}\,,\\[1mm]
	(1-\frac {u(x)}{c_1})/\mu& \text{if } x\in\overline{\Om}\,,\\[1mm]
	 -c/\mu & \text{if } x\in \Om'\,,
\end{cases}
\]
satisfies
\[
\int_{\RR^2} w\, (\De F+ \mu F)\, dx= \int_{\RR^2} F\rho\, dx\,,
\]
for any~$F\in C^\infty(\RR^2)$\footnote{Equivalently, $\De w +\mu w=\rho$ in the sense of compactly supported distributions.}. As rigid motions commute with the Laplacian, 
\[
\begin{aligned}
\int_{\RR^2} f\circ\cR^{-1}\, \rho\, dx & = \int_{\RR^2}  [\De (f\circ\cR^{-1})+ \mu f\circ\cR^{-1}]\,w\, dx\\
& = \int_{\RR^2} (\De f+\mu f)\circ\cR^{-1}\, w\, dx=0\,.
\end{aligned}
\]
The identity~\eqref{E.Pompclaim}, and therefore Corollary~\ref{C.Pompeiu}, follow.
\end{proof}

Before passing to the proof of Theorem~\ref{T.Euler}, let us recall that $(v,p)\in L^2_{\mathrm{loc}}(\RR^2,\RR^2)\times L^1_{\mathrm{loc}}(\RR^2)$ is a {\em  stationary weak solution}\/ of the incompressible Euler equations on~$\RR^2$ if
\[
\int_{\RR^2} v\cdot \nabla\phi\, dx=\int_{\RR^2} (v_i\,v_j\, \pd_i w_j+ p\Div w)\, dx=0\,,
\]
for all $\phi\in C^\infty_c(\RR^2)$ and all $w\in C^\infty_c(\RR^2,\RR^2)$. The first condition says that $\Div v=0$ in the sense of distributions, while the second condition (where summation over repeated indices is understood) is equivalent to the equation
\[
v\cdot \nabla v+\nabla p=0\,,
\]
when the divergence-free field $v$ and the function~$p$ are continuously differentiable.

\begin{proof}[Proof of Theorem~\ref{T.Euler}]
Take $\Om$, $u$ and $\mu$ as in Theorem~\ref{T.main}. We will also use the domain~$\Om'$ defined in the proof of Corollary~\ref{C.Pompeiu}. Since $\nabla u|_{\pd\Om}=0$, the vector field defined in Cartesian coordinates as
	\[
	v:=\begin{cases}
		(\frac{\pd u}{\pd x_2},-\frac{\pd u}{\pd x_1}) &\text{in } \Om\,,\\[1mm]
		0 & \text{in }\RR^2\backslash\Om\,,
	\end{cases} 
		\]
is continuous on~$\RR^2$, supported on~$\BOm$ and analytic in~$\Om$. It is also divergence-free in the sense of distributions, since
	\[
\int_{\RR^2}	v\cdot\nabla\phi\, dx=\int_{\Om}	v\cdot\nabla\phi\, dx=\int_{\Om}	\phi\, \Div v\, dx+ \int_{\pd\Om}	\phi\, v\cdot\nu \,dS= 0
	\]
	for all $\phi\in C^\infty_c(\RR^2)$. Here we have used that, by definition, $\Div v=0$ in~$\Om$ and $v|_{\pd\Om}=0$.
	
Now consider the continuous function
	\[
	p:=\begin{cases}
	0 & \text{on } \RR^2\backslash\overline{(\Om\cup\Om')},\\[1mm]
		-\frac12 (|\nabla u|^2+\mu u^2-\mu c_1^2) &\text{on } \BOm\,,\\[1mm]
		-\frac12\mu (c_2^2-c_1^2) & \text{on }\Om'\,.
	\end{cases} 
	\]
As an elementary computation using only the definition of~$(v,p)$ and the eigenvalue equation $\De u+ \mu u=0$ shows that
\[
v\cdot \nabla v+\nabla p=0 \quad \textup{in } \Omega\,,
\]	
one can integrate by parts to see that
	\[
	\begin{aligned}
	\int_{\RR^2} (v_i\,v_j\, \pd_i w_j+ p\Div w)\, dx &= \int_{\Om} (v_i\,v_j\, \pd_i w_j+ p\Div w)\, dx\\
	& = - \int_{\Om} (v\cdot \nabla v_j+\pd_j p)\,w_j\, dx =0\,,
	\end{aligned}
	\]
for all $w\in C^\infty_c(\RR^2,\RR^2)$. Thus $(v,p)$ is a  stationary weak solution of 2D Euler, and the theorem follows.
\end{proof}

\section{Rigidity results}\label{S.Rigidity}

In this section we prove two rigidity results highlighting that, as discussed in the Introduction, the question we analyze in this paper shares many features with the celebrated Schiffer conjecture. The first is a generalization of a classical theorem due to Berenstein~\cite{Berenstein} (which corresponds to the case $J=1$ in the notation we use below):

\begin{theorem}\label{T.many}
	Suppose that there is an infinite sequence of orthogonal Neumann eigenfunctions of a smooth bounded domain $\Om\subset\RR^2$ which are locally constant on the boundary. Then $\Om$ is either a disk or an annulus.
\end{theorem}

\begin{proof}
The proof uses arguments of \cite{Berenstein}, so we use a compatible notation for the reader's convenience. The proof, which relies on an asymptotic analysis employing the sequence of eigenvalues $\mu_n\to\infty$ and a geometric argument, consists of three steps.

\subsubsection*{Step 1: Initial considerations}
By hypothesis, there is an infinite sequence of distinct eigenvalues $\mu_n$ and corresponding Neumann eigenfunctions $u_n$ which are locally constant on the boundary. That is, $u_n|_{\Ga_j}=c^j_n$,
where $\Ga_j$, $1\leq j\leq J$ denote the connected components of~$\pd\Om$. Moreover, note that $\pd\Om$ is analytic by the results of Kinderlehrer and Nirenberg~\cite{KinderlehrerNirenberg}.

Let us call $\Om_1,\dots,\Om_{J-1}$ the bounded connected components of $\RR^2\backslash\BOm$ and relabel the boundary components if necessary so that $\pd\Om_j=\Ga_j$. For convenience, we set $\Om_J:=\Om$. A minor variation on the proof of Corollary~\ref{C.Pompeiu} shows that there exists a function $w_n\in C^1_c(\RR^2)$, equal to $a_n u_n+b_n$ in~$\BOm$ for some $a_n,b_n\in\RR $ and constant in $\Om_j$ for $1\leq j\leq J-1$,  which satisfies the equation
\begin{equation} \label{ber1}
\De w_n+\mu_n w_n=\sum_{j=1}^J \ga_{n,j} \mathbbm 1_{\Om_j}=:\rho_n,
\end{equation}
for some real constants $\ga_{n,j}$. The above identity has to be understood in the sense of compactly supported distributions (that is, $\int_{\RR^2} w_n (\De F+ \mu_n F)\, dx= \int_{\RR^2} \rho_nF\, dx$ for all $F\in C^\infty(\RR^2)$). Note that $\rho_n$ is not the zero function.

We now consider the Fourier-Laplace transform $\hrho_n: \CC^2 \to \CC$ of~$\rho_n$, defined as
\[
 \hrho_n(\zeta):=\int_{\RR^2}e^{-i\zeta\cdot x}\, \rho_n(x)\, dx.
\]
Since $\rho_n$ is compactly supported, $\hrho_n$ is well defined in $\CC^2$ and of class $C^{\infty}$. By \eqref{ber1}, $\hrho_n$ must vanish identically on the set 
\[
\cC_n:=\{\zeta\in\CC^2: \zeta_1^2+\zeta_2^2=\mu_n\}\,.
\]

Let us now use complex variables $z:=x_1+ i x_2$, $\bar z:=x_1-i x_2$. We consider $\si_n:=\pd_{\bar z}\rho_n$, with the obvious notation, and observe that
$$ \hsi_n(\zeta) = \frac{i}{2} (\zeta_1 + i \zeta_2) \widehat{\rho}_n(\zeta).$$
Hence $\hsi_n$ also vanishes in the sets $\cC_n$. 

The $\bar z$-derivative of an indicator function is explicitly computed in~\cite[Equation (21)]{Berenstein} using the Cauchy integral formula. Specifically, as $\Om_j$ is simply connected for $j\leq J-1$, we have
\[
\pd_{\bar z} \mathbbm 1_{\Om_j}=\frac i2 \chi_{\Ga_j}
\]
for $1\leq j\leq J-1$, where $\chi_{\Ga_j}$ is the compactly supported distribution defined by
\[
\langle \chi_{\Ga_j}, \phi\rangle := \int_{\Ga_j}\psi(x_1,x_2)\, (dx_1+i \,dx_2)\,,\qquad \psi\in C^\infty_c(\RR^2)\,.
\]
Likewise, the orientation of the various boundary components implies that
\[
\pd_{\bar z} \mathbbm 1_{\Om_J}=\frac i2 \left(\chi_{\Ga_J}-\sum_{j=1}^{J-1} \chi_{\Ga_j}\right)\,.
\]
Therefore, there are some real constants~$\si_{n,j}$ such that
\begin{equation}\label{E.hsi}
\hsi_n=\frac i2\sum_{j=1}^J \si_{n,j}\, \widehat{\chi_{\Ga_j}}\,.	
\end{equation}
The key feature of this expression is that the constants $\si_{n,j}$ depend on~$n$ but the distributions $\chi_{\Ga_j}$ do not.

Multiplying $w_n$ by a constant if necessary, we can assume that
\[
\max_{1\leq j\leq J}|\si_{n,j}|=1
\]
for all~$n$ in the sequence. By compactness, one can therefore pass to a subsequence of the eigenvalues $\mu_n$ (which we will still denote by $\mu_n$ for convenience) so that, along this subsequence,
\[
\lim_{n\to\infty}\si_{n,j}=\si_j
\]
for all $1\leq j\leq J$, where 
\begin{equation}\label{E.normalsi}
\max_{1\leq j\leq J}|\si_j|=1\,.
\end{equation}
Replacing $w_n$ by $-w_n$ if necessary, we can assume without loss of generality that $\si_{j_0}=1$ for some $1\leq j_0\leq J$.

\subsubsection*{Step 2: Geometric considerations}

For any point $p \in \partial \Omega$, we denote by $\xi(p)\in\RR^2$ its outer normal vector and $\eta(p)\in\RR^2$ its tangential vector (defined by rotating~$\xi(p)$ ninety degrees counterclockwise, for concreteness). For each point $p \in \pd\Om$, let us introduce the notation
\[
\cJ(p):=\big\{ q\in\pd\Om: \xi(q)=\pm\xi(p),\; (q-p)\cdot \eta(p)=0\big\}\,.
\]
In words, $\cJ(p)$ is the set of points on the straight line $ \{p + t \xi(p): t \in \RR \}$ which intersect  $\partial \Omega$ orthogonally. Obviously, $p \in \cJ(p)$, and if $q \in \cJ(p)$, then $\cJ(q)=\cJ(p)$.

Since this intersection is orthogonal, for each point in $\cJ(p)$ there is a neighborhood of $p$ in~$\pd\Om$ that does not contain any other points of $\cJ(p)$, so the compactness of~$\pd\Om$ ensures that $\cJ(p)$ is a finite set for each~$p$. 
Also by compactness, we conclude that there exists some $\ep>0$ such that
\begin{equation}\label{compact} 
\min_{p\in\pd\Om}\; \min_{q, q'\in\cJ(p)}|q-q'|>2\ep\,,
\end{equation}
where of course we assume that $q\neq q'$.

Now let
$$
m := \min_{p \in \Gamma_{j_0}} \# \cJ(p)
$$
the minimal number of elements in $\cJ(p)$ as $p$ takes values in the boundary component~$\Ga_{j_0}$. We shall next prove the following auxiliary result:

\begin{lemma}\label{L.open}
	The set of points $p \in \Gamma_{j_0}$ such that $\cJ(p)$ has exactly $m$ elements is open. 
\end{lemma}

\begin{proof}
Take any point $p_1 \in \Gamma_{j_0}$ such that $ \cJ(p_1)=\{p_1, p_2, \dots, \, p_m\}$ consists of~$m$ points. By continuity, for each $q\in \Gamma_{j_0}$ close enough to $p_1$, the elements of $\cJ(q)$ are necessarily contained in $\bigcup_{i=1}^m B(p_i, \ep)$, where $\ep$ is as in~\eqref{compact}. By \eqref{compact}, there exists at most one element of $\cJ(q)$ in each ball $B(p_i, \ep)$, and hence $\#\cJ(q) \leq m$. Since $\cJ(q)$ must have at least $m$ elements, we conclude that $\#\cJ(q)=m$.
\end{proof}

\begin{remark}
	In the proof we have in fact obtained a stronger result: in each ball $B(p_i, \ep)$, there exists exactly one element of $\cJ(q)$ if $q$ is close to $p_1$. 
\end{remark}

In addition to the set $\cJ(p)$, for each $p\in\pd\Om$ we also consider the set
\[
\tcJ(p):=\big\{ q\in\pd\Om: \xi(q)=\pm\xi(p)\big\}
\]
of boundary points whose normal is parallel to the normal at~$p$. By the analyticity of the bounded domain~$\Om$, $\tcJ(p)$ is a finite set. Obviously $\cJ(p)\subset\tcJ(p)$.

Observe that, as $\pd\Om$ is analytic, its curvature can only vanish at finitely many points. Hence, by Lemma~\ref{L.open} we can take a point $p_1\in\Ga_{j_0}$ such that $\cJ(p_1)= \{p_1, \, p_2 \dots , p_m\}$ consists of~$m$ points  and such that $k(p)\neq0$ for all $p\in\tcJ(p_1)$. Here and in what follows, $k(p)$ denotes the curvature of the boundary at a point~$p$. We can compute this curvature with respect to the normal $\pm\xi$, so that $k(p_1)>0$. We can choose our Cartesian coordinates so that $\xi(p_1)=(1,0)$ and write the points in $\tcJ(p_1)$ as $\tcJ(p_1)=\{p_1,\dots, p_M\}$ with $M\geq m$.

Since the curvature does not vanish at~$p_1$, the Gauss map in locally invertible in a neighborhood of $p_1$ in $\Gamma_{j_0}$. For~$s$ in a small neighborhood of~$0$, we shall then denote by $p_1(s)$ the only point in a neighborhood of~$p_1$ for which $\xi(p_1(s))=\txi(s)$, where $\txi(s):=(\cos s,\sin s)$. We also introduce the shorthand notation $\tk(s):=k(p_1(s))$. Note that $p_1(0)=p_1$.

Applying the same reasoning to the other points $p_i\in\tcJ(p_1)$, with $1\leq i\leq M$, we obtain the corresponding maps $p_i(s)$ with $p_i(0)=p_i$. By the local invertibility of the Gauss map,
\[
\tcJ(p_1(s))=\{ p_1(s),\dots, p_M(s)\}
\]
for all sufficiently small~$s$, say $|s|<s_0$. 
Furthermore, we can argue as in the proof of Lemma~\ref{L.open} to conclude that
\[
\cJ(p_1(s))= \{p_1(s), \, p_2(s), \ \dots p_m(s)\}\,.
\]
Indeed, for each $1\leq i\leq m$, these are the only points in a neighborhood of $p_i$ whose normal is parallel to $\txi(s)$ by the local invertibility of the Gauss map around $p_i$, and all of them necessarily belong to $\cJ(p_1(s))$ because this set must have at least~$m$ elements.

We can now argue as in \cite[Equation (43)]{Berenstein} to conclude that, for small~$s$, the curves $p_i(s)$ are parallel. That is, there exist constants $C_i \in \RR$ such that
$$ p_i(s)= p_1(s) + C_i\, \xi(p_1(s))$$
for $1\leq i\leq M$. Furthermore, $C_i \neq C_j$ if $i \neq j$ (because $p_i(s)\neq p_j(s)$) and $C_1=0$. Since the curves are parallel, their curvatures are related as
\begin{equation} \label{curvature} k(p_i(s)) = \pm \frac{\tk(s)}{1 + C_i \tk(s)}\,, \end{equation}
where the sign $\pm$ depends on the orientation, that is, on whether $\xi(p_i(s))$ coincides with $\txi(s)$ or with $-\txi(s)$. As the boundary is regular,  that $1 + C_i \tilde{k}(s)$ cannot vanish for small $s$.

\subsubsection*{Step 3: Asymptotic analysis and conclusion of the proof}

Following~\cite[Equations (23)--(25)]{Berenstein},  for~$s$ in a neighborhood of~0, let us now set
\[
\zeta(s):=r\,\txi(s) + it\, \teta(s)\,,
\]
where $\teta(s):=(-\sin s,\cos s)$ is obtained by rotating $\txi(s)$ ninety degrees counterclockwise. Here $t$ is any real number with $1\ll |t|\ll \log\mu_n$ (for large~$\mu_n$) and $r:= (\mu_n+t^2)^{1/2}$. Note that $\zeta(s)\in \cC_n$ for all small~$s$, so necessarily
\begin{equation}\label{E.hsin0}
\hsi_n(\zeta(s))=0
\end{equation}
for all $t$ as above and all small~$s$. Note the smallness condition $|s|<s_0$ is purely geometric, so $s_0$ is independent of~$n$.

In this regime, the asymptotic behavior of $\widehat{\chi_{\Ga_j}}(\zeta(s))$ for large~$|t|$ was computed in~\cite[Lemma 2]{Berenstein}, which is the key result used in~\cite{Berenstein}. This immediately results in an asymptotic formula for $\hsi_n(\zeta(s))$, namely
\begin{equation}\label{E.asympthsi}
	\hsi_n(\zeta(s))=\left(\frac{2\pi}{r}\right)^{1/2}\left[\sum_{\ell=1}^M\tsi_{n,\ell}\,U_\ell(r,s) |k(p_\ell(s))|^{-\frac12}e^{t p_\ell(s)\cdot\teta(s)}  +o(1)\right]\,,
\end{equation}
which must vanish identically by~\eqref{E.hsin0}.
Here $\tsi_{n,\ell}:=\si_{n,j}$, where $\Ga_j$ is the boundary component on which $p_\ell$ lies, and 
\[
U_\ell(r,s):=\eta^{\CC}(p_\ell(s))\,e^{-i\frac\pi4\sign k(p_\ell(s))}e^{-irp_\ell(s)\cdot \txi(s)}
\]
is a complex number of unit modulus, with
\[
\eta^{\CC}(p_\ell(s)):=\eta_1(p_\ell(s))+i\eta_2(p_\ell(s))\,,
\]
and $o(1)$ denotes a term that tends to~$0$, uniformly for $|s|<s_0$, as $n\to\infty$.
For future reference, we similarly set $\tsi_\ell:= \si_j$, with $j$ as above. Here and in what follows, $(\eta_1,\eta_2)$ denote the components of the vector field~$\eta$. Of course, $\eta(p_\ell(s))$ coincides with $\teta(s)$ modulo a sign.

Since the tangent vectors at all points $p_i$ are parallel, let us define the complex number of unit modulus
\[
\de_\ell:=\xi(p_1)\cdot \xi(p_\ell)\, e^{-i\frac\pi4 \sign k(p_\ell)}
\]
for all $1\leq \ell \leq M$, which enables us to write
\[
e^{-i\frac\pi4 \sign k(p_\ell(s))}\eta^{\CC}(p_\ell(s))= \de_\ell \, \eta^{\CC}(p_1(s))
\]
Also, for each fixed~$|s|<s_0$, consider the set
\[
\cB(s):=\{ p_\ell(s)\cdot \teta(s): 1\leq \ell\leq M\}
\]
of {\em distinct}\/ values of $p_\ell(s)\cdot \teta(s)$, and let us set
\[
\cI_b(s):=\big\{\ell\in\{1,\dots,M\}: p_\ell(s)\cdot \teta(s)=b\big\}
\]
for each $b\in\cB(s)$. Note that, by the definition of the set $\tcJ(p_1(s))$, 
\[
\cI_{b_1(s)}(s)=\{1,\dots, m\}
\]
if we set 
\begin{equation}\label{E.b1}
b_1(s):= p_1(s)\cdot \teta(s)\,.	
\end{equation}
By continuity, we can assume that $\cB(s)\subset\RR\backslash\{0\}$.

Using that $\tsi_{n,\ell}=\tsi_\ell+o(1)$, Equation~\eqref{E.asympthsi} can then be rewritten in a more convenient way as
\begin{align*}
0 &=   \hsi_n(\zeta(s)) \\
& = \left(\frac{2\pi}{r}\right)^{1/2}\bigg[\eta^{\CC}(p_1(s))\sum_{b\in\cB(s)} e^{tb}\sum_{\ell\in\cI_b(s) } \de_\ell[\tsi_\ell+o(1)] |k(p_\ell(s))|^{-\frac12}e^{-ir\, p_\ell(s)\cdot \txi(s)} \\
	& \hspace{10.8cm }+o(1)\bigg]\,.
\end{align*}
As shown in~\cite[Equation (41)]{Berenstein}, for any fixed~$|s|<s_0$ this implies that
\[
\sum_{\ell\in\cI_b(s) } \de_\ell\tsi_\ell |k(p_\ell(s))|^{-\frac12}e^{-ir\, p_\ell(s)\cdot \txi(s)}=o(1)\,,
\]
where the term $o(1)$ tends to~0 as $n\to\infty$,
for all $b\in\cB(s)$. Thus, taking $b=b_1(s)$ as in~\eqref{E.b1}, we find
\[
\sum_{\ell=1}^m \de_\ell\tsi_\ell |k(p_\ell(s))|^{-\frac12}e^{-ir\, p_\ell(s)\cdot \txi(s)}=o(1)\,,
\]
for each fixed~$s$. Note that the coefficients $\tsi_\ell$ are not zero for all $1 \leq \ell \leq m$ because $\tsi_1=1$.

Differentiating $p_\ell(s)\cdot \txi(s)$ with respect to~$s$, for each $1\leq \ell\leq m$ we find
\[ 
 \frac{\rm d}{{\rm d}s} \left( p_\ell(s)\cdot \txi(s) \right)= \frac {{\rm d}p_\ell}{{\rm d}s} \cdot \txi(s)+ p_\ell(s)\cdot \frac {{\rm d}\txi}{{\rm d}s} = p_\ell(s)\cdot \teta(s)= p_1(s)\cdot \txi(s)\,,
\]
where we have used that $p_\ell(s)\in \cJ(p_1(s))$. Thus we have
\begin{equation}\label{E.simpleeq}
\begin{aligned}
o(1)& =\sum_{\ell=1}^m \de_\ell\tsi_\ell |k(p_\ell(s))|^{-\frac12}e^{-ir\, p_\ell(s)\cdot \txi(s)} \\
& = e^{-ir\int_0^s p_1(s')\cdot \txi(s')\, ds'}\sum_{\ell=1}^m \tde_\ell |k(p_\ell(s))|^{-\frac12}
\end{aligned}
\end{equation}
for all~$|s|<s_0$, with $r$-dependent complex numbers
\[
\tde_\ell:= \de_\ell\,\tsi_\ell\, e^{ir\, p_\ell\cdot \txi(0)}
\]
not all of which are zero. Since the factor multiplying the sum in~\eqref{E.simpleeq} is of unit modulus, using the formula~\eqref{curvature} for the curvature, we thus infer that
\begin{equation}\label{E.gk}
\sum_{\ell=1}^m \tde_\ell \left| \frac{\tk(s)}{1+C_\ell\tk(s)}\right|^{-\frac12}=o(1)
\end{equation}
for all $|s|<s_0$.

Now note that the norm of the complex vector
\[
\widetilde A_r:= (\tde_1,\dots, \tde_m)\in\CC^m
\]
satisfies $|\widetilde A_r|=c_0>0$ for all~$r$, so we can pass to a subsequence where the real or the imaginary part of~$A_r$ has norm larger than $c_0/2$. Multiplying $\widetilde A_r$ by the imaginary unit if necessary, we can thus assume that the real vector 
\[
 (\Re \tde_1,\dots, \Re \tde_m)\in\RR^m
\]
satisfies $c_0/2\leq |A_r|\leq c_0$.

Denoting by
\[
A_r:=\frac{(\Re \tde_1,\dots, \Re \tde_m)}{[(\Re \tde_1)^2+ \cdot + (\Re \tde_m)^2]^{1/2}}\in\mathbb{S}^{m-1}
\]
the corresponding unit vector, by the compactness of~$\SS^{m-1}$ one can assume that $A_r$ converges to some $A\in\mathbb{S}^{m-1}$ along a certain subsequence $r\to\infty$ (which do not reflect notationally). Equation~\eqref{E.gk} can then be rewritten
in terms of the $r$-independent vector field on~$\RR^m$
\[
V(s):=\left(\frac{\tk(s)}{1+C_1\tk(s)},\dots, \frac{\tk(s)}{1+C_m\tk(s)} \right)
\]
and the unit vector~$A_r$
as
\begin{equation*}
A_r\cdot V(s)=o(1)
\end{equation*}
for all $|s|<s_0$. We thus infer that
\begin{equation}\label{E.ArV}
A\cdot V(s)=0\,.
\end{equation}
Geometrically, this means that the vector field $V(s)$ must be tangent to a hyperplane that passes through the origin. 
As
$
\{V(s): |s|<s_0\}\subset\RR^m
$
is not contained in a hyperplane unless $\tk(s)$ is constant, we conclude from~\eqref{E.ArV} that $\tk(s)=\tk(0)$ for all $|s|<s_0$. This means that the curve $p_1(s)$, which lies in the boundary component~$\Ga_{j_0}$, is an arc of a circle of radius $R:=1/\tk(0)$. By analyticity, $\Ga_{j_0}$ is then a circle. Without any loss of generality, we can assume this circle is centered at the origin. That is, $\Ga_{j_0}=\{r= R\}$, where $r$ is the radial coordinate.

This readily implies that the domain is a disk or an annulus. To see this, consider the region
\[
\Om':=\{x\in\Om: \dist(x,\Ga_{j_0})<\ep\}\,,
\]
which is necessarily a thin annulus if $\ep $ is small enough, and one of the eigenfunctions $u_n$ in our sequence, whose Neumann eigenfunction is~$\mu_n$ and which takes some constant value $c$ on~$\Ga_{j_0}$. Let $\cR$ be any rotation matrix and set $u_n^{\cR}(x):= u_n(\cR x)$. Now note that $u_n$ and $u_n^{\cR}$ satisfy the equation
\[
\Delta u + \mu_n u=0 \qquad \text{in }\Om'
\]
and the same boundary conditions at $r=R$:
\[
u|_{r=R}=c\,,\qquad \pd_r u|_{r=R}=0\,.
\]
By unique continuation, $u_n=u_n^{\cR}$ for any rotation~$\cR$, so $u_n$ is a radial function. Since $\Om$ is connected and $u_n$ is constant on each boundary component, the domain~$\Om$ must then be a disk or an annulus. The theorem then follows.
\end{proof}

The second rigidity result we want to show is the analogous to~\cite{aviles, Deng} and ensures that if the eigenfunction corresponds to a sufficiently low eigenvalue, no nontrivial domains exist. It is interesting to stress that in our main theorem the nontrivial domains with Neumann eigenfunctions that are locally constant on the boundary correspond to the 18th lowest Neumann eigenfunction of the domain in the case $l=4$, as shown in Proposition \ref{l=4} in Appendix \ref{A.transversality}. 

Let us denote by $\mu_k(\Om)$ (respectively, $\la_{k}(\Om)$) the $k$-th Neumann (respectively, Dirichlet) eigenvalue of a planar domain~$\Om$, counted with multiplicity and with $k=0,1,2,\dots$ We have the following:
\begin{theorem} \label{T.Aviles} Let $\Omega \subset \mathbb{R}^2$ be a smooth bounded domain such that the problem
$$
\Delta u + \mu u = 0\quad in\ \Omega\,, \qquad \nabla u = 0 \quad on \ \partial\Omega\,,
$$
%	$$ \left \{ \begin{array}{ll} \Delta u + \mu u =0, & x \in \Omega, \\ \nabla u=0, & x \in \partial \Om.\end{array} \right.$$
admits a nonconstant solution. Then:

\begin{enumerate}
	\item If $\mu \leq \mu_4(\Om)$,  $\Omega$ is either a disk or an annulus.
	
	\item If $\partial \Omega$ has exactly two connected components and $\mu \leq \mu_5(\Om)$,  $\Omega$ is an annulus.
\end{enumerate}
\end{theorem}

\begin{proof} Let us start with item~(i). Consider the functions
$$ 
\phi_1:= {\partial_1 u}\,, \qquad \phi_2:= {\partial_2 u}\,, \qquad \phi_3:= x_1 {\partial_2 u} - x_2 {\partial_1 u}\,,
$$ 
where $\pd_j$ denotes the derivative with respect to the Cartesian coordinate~$x_j$, and observe that $\phi_1$ and $\phi_2$ correspond to the infinitesimal generator of translations, whereas $\phi_3$ is related to rotations. Let us also define
$$ 
E := \mbox{span } \big\{ \phi_1,  \phi_2, \phi_3 \big\}\,.
$$

Since $\phi_i$ commute with the Laplacian and $\nabla u=0$ on the boundary, it is clear that the functions $\phi_i$ are Dirichlet eigenfunctions of the Laplace operator with eigenvalue $\mu$. If $\phi_1$ and  $\phi_2$ are proportional, this would imply that $u$ is constant on parallel lines. Taking into account that $\Omega$ is bounded we would get that $u$ is constant, contradicting our hypotheses.

As a consequence, the multiplicity of $\mu$ as a Dirichlet eigenvalue is at least $2$, which implies that $\mu = \lambda_k(\Om)$ with $k \geq 1$. At this point, we consider separately two different cases:
\begin{enumerate} 
	\item[(i.1)] {\em The dimension of $E$ is three.}\/ 
In this case, $\mu$ has at least multiplicity~$3$ as a Dirichlet eigenvalue. That is, for some $k \geq 1$, 
$$
 \mu = \lambda_k(\Om)= \lambda_{k+1}(\Om)= \lambda_{k+2}(\Om)\,.
$$
We now make use of \cite{Fi, Fri} to get $\lambda_{k+2}(\Om) > \mu_{k+3}(\Om)$, which is in contradiction with $\mu \leq \mu_4(\Om)$. 
\smallskip

\item[(i.2)] {\em The dimension of $E$ is two.}\/  That is, the functions $\{\phi_i\}_{i=1}^3$ are linearly dependent. In this case, there exist constants $x_1^0$, $x_2^0 \in \RR$ such that
$$ 
(x_1-x_1^0) {\partial_2 u} - (x_2-x_2^0) {\partial_1 u}=0\,.
$$
This implies that $u$ is radially symmetric with respect to the point $x^0:=(x_1^0, x_2^0)$. Since $\Omega$ is connected and $u$ is constant on $\partial \Omega$, we conclude that $\Omega$ is either a disk or an annulus.
\end{enumerate}

\medbreak
Let us now move on to item~(ii). The proof considers two different cases:

\subsubsection*{Case 1: The function $u$ has no critical points in $\Omega$.}
Indeed, if $u$ has no critical points, its global minimum and maximum are attained at the boundary. Assume for concreteness that
$$ u=M = \mbox{max} \, u\ \mbox{ in } \Gamma_1\,, \quad  u= m = \mbox{min} \, u\ \mbox{ in } \Gamma_2\,,$$
where $\Gamma_1$ and $\Gamma_2$ are the inner and outer connected component of $\partial \Omega$, respectively. We define the function $v:= u-m \geq 0$, which satisfies
$$
\Delta v + \mu (v+m) = 0 \quad\textup{in } \Omega\,, \qquad  \nabla v = 0 \quad \textup{on } \partial \Omega\,,
$$
with  $v = 0$ on $\Gamma_2$ and $ v = M-m$ on $\Gamma_1$. A result of Reichel~\cite{Reichel} then shows that $\Omega$ is an annulus.

\subsubsection*{Case 2: The function $u$ has at least a critical point $p \in \Omega$. }

We argue as in the proof of (i). If the dimension of $E$ is two, we are done, so we assume that the dimension of $E$ is three. Reasoning as in the proof of (i), we obtain that $\mu > \mu_{k+3}(\Om)$. Hence, if we show that $k \geq 2$, we are done. In other words, we need to prove that $\mu \neq \lambda_1(\Om)$.  

Observe that $\phi_i(p)=0$ for any $i=1, 2,  3$. Let us then consider the equation
$$ \nabla \phi(p)=0, \qquad \phi \in E\,.
$$
This is a linear system of two equations posed in a vector space of dimension three, so it admits a nonzero solution $\phi_0 \in E$. As a consequence, it follows that
\begin{equation} \label{nodal}
\phi_0(p)=0\,, \quad \nabla \phi_0(p)=0\,.
\end{equation}

We now claim that this is impossible for $\mu= \lambda_1(\Om)$. This fact is surely known, but we have not been able to find a explicit reference in the literature. The closest result that we have found is~\cite[Theorem 3.2, Corollary 3.5]{cheng}, in the setting of closed surfaces. However, the argument there cannot be directly translated to the setting of domains with boundary. For the sake of completeness, we give a short proof of this claim.

First of all, we define 
$$
\begin{aligned}
& N:= \{x \in \Omega: \ \phi_0(x)=0\}\,, \\ 
& \Omega^+:= \{x \in \Omega:\ \phi_0(x)>0\}\,, \\
& \Omega^-:= \{x \in \Omega:\ \phi_0(x)<0\}\,.
\end{aligned}
$$ 
Note that if $\mu= \lambda_1(\Omega)$, the Courant nodal theorem implies that the sets $\Omega^{\pm}$ are connected. 

On the other hand, it is known (see e.g. \cite{cheng}) that \eqref{nodal} implies the existence of $j$ analytic curves in $N$ ($j\geq 2)$ that intersect transversally at $p$. Moreover, $\phi_0$ must change sign across each curve. Then, if $r$ is sufficiently small, $B_p(r) \setminus N$ has $2j$ connected components $\omega_i^{\pm} \subset \Omega^{\pm}$, $i=1 \dots j$. We take $q_i^{\pm}$ points in $\omega_i^{\pm}$. 

Now, choose two points $q_i^+$, $q_{i'}^+$, $i \neq i'$. In the case where $\mu = \lambda_1(\Omega)$, it follows that $\Omega^+$ is connected, so there exists a curve in $\Omega^+$ joining these two points. We can also take a curve connecting  $q_i^+$ to $p$ in $B_p(r) \setminus N$, and the same for $q_{i'}^+$. Joining those curves, we obtain a closed curve $\gamma^+$ such that $\mbox{supp}\ \gamma^+ \subset \Omega^+ \cup \{p\}$. In the same way, we can build $\gamma^-$ such that $\mbox{supp} \ \gamma^- \subset \Omega^- \cup \{p\}$.  

Observe that $\gamma^+$ and $\gamma^-$ intersect only at the point $p$. By choosing the departing points apropriately, this intersection is transversal, which is impossible. Hence, if $\mu = \lambda_1(\Omega)$, we have a contradiction, and the claim follows.
\end{proof}

\begin{remark} One can probably refine the arguments to obtain a better estimate on the value of~$k$ of $\mu_k(\Om)$ for which no nontrivial domains exist, in the spirit of \cite{aviles, Deng}. Finding the optimal value of~$k$ seems hard.
\end{remark}

%\begin{remark}
%Berenstein's argument does not directly show that all the boundary components must be circles. Indeed, the argument hinges on an asymptotic expansion for $\widehat{\chi_{\Ga_j}}$ (for complex frequencies in a certain region) in which the error is $o(1)$ as $n\to\infty$. The normalization~\eqref{E.normalsi} ensures the argument yields useful information at least in the case of points on~$\Ga_{j_0}$, but not necessarily about points on another boundary component $\Ga_j$ if the constant~$\si_{n,j}$ is of order $o(1)$ as~$n\to\infty$.
%\end{remark}

\section*{Acknowledgements}
This work has received funding from the European Research Council (ERC) under the European Union's Horizon 2020 research and innovation programme through the grant agreement~862342 (A. E.). A. E. is also supported by the grant PID2022-136795NB-I00 of the Spanish Science Agency and the ICMAT--Severo Ochoa grant CEX2019-000904-S. D. R. has been supported by:
the Grant PID2021-122122NB-I00 of the MICIN/AEI, the \emph{IMAG-Maria de Maeztu} Excellence Grant CEX2020-001105-M funded by MICIN/AEI, and the Research Group FQM-116 funded by J. Andalucia. P. S. has been supported by the Grant PID2020-117868GB-I00 of the MICIN/AEI and the \emph{IMAG-Maria de Maeztu} Excellence Grant CEX2020-001105-M funded by MICIN/AEI.

\appendix

\section{The transversality condition for $l=4$}\label{A.transversality}

Our aim in this Appendix is the verification of condition \eqref{T} for $l=4$, which will ensure that the statement of Theorem~\ref{T.main} (or Theorem~\ref{T.mainDetails}) is valid for the particular case $l=4$. In this case, we can also compute the order of the eigenvalue $\mu$ for which one can find nonradial Neumann eigenfunctions that are locally constant on the boundary. Recall that we denote by $\mu_k(\Om)$ the $k$-th Neumann eigenvalue (counted with multiplicity) of a planar domain~$\Om$, with $k=0,1,2,\dots$

\begin{proposition} \label{l=4} If $l=4$, condition \eqref{T} holds. Moreover,  $\mu_{0,2}(a_4)= \mu_{18}(\Om_{a_4})$.
\end{proposition}

\begin{proof}

Recall that $\mu\equiv \mu_{0,2}(a)$ is the value such that the following equation has a nontrivial solution, corresponding to the second  eigenfunction
\begin{equation*}
	\psi''+\frac{\psi'}r +\mu\, \psi=0\quad \text{in } (a,1)\,,\qquad \psi'(a)=\psi'(1)=0\,.
\end{equation*}
The function $\psi$ is the given by
$$ \psi(r)= \frac{J_0(\sqrt{\mu}r) Y_1(\sqrt{\mu}) - Y_0(\sqrt{\mu}r) J_1(\sqrt{\mu})}{Y_1(\sqrt{\mu})}\,.
$$
Here and in what follows, $Y_k$ denotes the Bessel function of the second kind of order $k$. Observe that $\psi$ changes sign twice in the interval $[a,1]$. Furthermore, $\psi'$ can be easily computed:
$$ \psi'(r)= -\sqrt{\mu}\, \frac{J_1(\sqrt{\mu}r) Y_1(\sqrt{\mu}) - Y_1(\sqrt{\mu}r) J_1(\sqrt{\mu})}{Y_1(\sqrt{\mu})}\,.$$
Thus, we will define, for any real $\mu$ and $r$,  
$$ \Psi(\mu, r):= \frac{J_1(\sqrt{\mu}r) Y_1(\sqrt{\mu}) - Y_1(\sqrt{\mu}r) J_1(\sqrt{\mu})}{Y_1(\sqrt{\mu})}\,.$$
Note that $\Psi(\mu_{0,2}(a), a) =0$. The chain rule then yields the following formula for $\mu_{0,2}'(a)$:
\begin{align*} \mu'_{0,2}& (a) = - \frac{\partial_r \Psi(\mu,r)}{\partial_{\mu} \Psi(\mu,r)}\Big|_{\mu=\mu_{0,2}(a), \, r=a} = \frac{P_{\mu}}{Q_{\mu}}\,,
\end{align*}
with
\begin{multline*}
P_{\mu} := \mu^{3/2} \pi Y_1(\sqrt{\mu}) \Big [ \ Y_1(\sqrt{\mu}) (J_2(\sqrt{\mu}a) - J_0(\sqrt{\mu}a)) \\
- J_1(\sqrt{\mu}) (Y_2(\sqrt{\mu}a) - Y_0(\sqrt{\mu}a))\Big]\,,
\end{multline*}
and 
\begin{multline*}
Q_{\mu} := 2  Y_1(\sqrt{\mu}a) + \pi Y_1(\sqrt{\mu}) \Big [ \ Y_1(\sqrt{\mu})(\sqrt{\mu} a  J_0(\sqrt{\mu}a)- J_1(\sqrt{\mu}a)) \\  - J_1(\sqrt{\mu})  (\sqrt{\mu} a Y_0(\sqrt{\mu}a) - Y_1(\sqrt{\mu}a))\,\Big ]\,.
\end{multline*}
Note that $\mu \equiv \mu_{0,2}(a)$ in the definitions of $P_{\mu}$ and $Q_{\mu}$. 
%\\  &=  \frac{\mu^{3/2} \pi Y_1(\sqrt{\mu}) \Big [Y_1(\sqrt{\mu}) (J_2(\sqrt{\mu}a) - J_0(\sqrt{\mu}a)) - J_1(\sqrt{\mu}) (Y_2(\sqrt{\mu}a) - Y_0(\sqrt{\mu}a))\Big]}{2  Y_1(\sqrt{\mu}a) + \pi Y_1(\sqrt{\mu}) \Big [Y_1(\sqrt{\mu})(\sqrt{\mu} a  J_0(\sqrt{\mu}a)- J_1(\sqrt{\mu}a)) - J_1(\sqrt{\mu})  (\sqrt{\mu} a Y_0(\sqrt{\mu}a) - Y_1(\sqrt{\mu}a)) \Big ]}\,,	
%\end{align*}
%where $\mu\equiv \mu_{0,2}(a)$ in the second line.

On the other hand, $\lambda\equiv  \lambda_{l,0}(a)$ is the value such that the following equation has a nontrivial solution corresponding to its principal eigenfunction:
\begin{equation*}
	\vp''+\frac{\vp'}r-\frac{l^2 \vp}{r^2}+\la\, \vp=0\quad \text{in } (a,1)\,,\qquad \vp(a)=\vp(1)=0\,.
\end{equation*}
The function $\vp$ is then given by
$$ \vp(r)= \frac{J_l(\sqrt{\la}r) Y_l(\sqrt{\la}) - Y_l(\sqrt{\la}r) J_l(\sqrt{\la})}{Y_l(\sqrt{\la})}\,.$$
As before, we then define, for arbitrary $r$ and $\la$, 
$$\Phi(\la,r):= \frac{J_l(\sqrt{\la}r) Y_l(\sqrt{\la}) - Y_l(\sqrt{\la}r) J_l(\sqrt{\la})}{Y_l(\sqrt{\la})}\,.$$
Again we have that $\Phi(a, \la_{l,0}(a))=0$, and hence we can compute $\la_{l,0}'(a)$ as:
\begin{align*} \la'_{l,0}&(a)= - \frac{\partial_r \Phi(\la,r)}{\partial_{\la} \Phi(\la,r)}\Big|_{\la=\la_{l,0}(a), \, r=a} = \frac{P_{\la}}{Q_{\la}}\,,
\end{align*}
with
\begin{multline*}
P_{\la} := \la^{3/2} \pi Y_l(\sqrt{\la}) \Big [Y_l(\sqrt{\la}) (J_{l+1}(\sqrt{\la}a) - J_{l-1}(\sqrt{\la}a)) \\ - J_l(\sqrt{\la}) (Y_{l+1}(\sqrt{\la}a) - Y_{l-1}(\sqrt{\la}a))\Big]\,,
\end{multline*}
and 
\begin{multline*}
Q_{\la}:= 2  Y_l(\sqrt{\la}a) + \pi Y_l(\sqrt{\la}) \Big [  Y_l(\sqrt{\la})(\sqrt{\la} a  J_{l-1}(\sqrt{\la}a)- l J_l(\sqrt{\la}a))\\
- J_l(\sqrt{\la})  (\sqrt{\la} a Y_{l-1}(\sqrt{\la}a) -l  Y_l(\sqrt{\la}a)) \Big ] 
\end{multline*}
%\\	&= \frac{\la^{3/2} \pi Y_l(\sqrt{\la}) \Big [Y_l(\sqrt{\la}) (J_{l+1}(\sqrt{\la}a) - J_{l-1}(\sqrt{\la}a)) - J_l(\sqrt{\la}) (Y_{l+1}(\sqrt{\la}a) - Y_{l-1}(\sqrt{\la}a))\Big]}{2  Y_l(\sqrt{\la}a) + \pi Y_l(\sqrt{\la}) \Big [Y_l(\sqrt{\la})(\sqrt{\la} a  J_{l-1}(\sqrt{\la}a)- l J_l(\sqrt{\la}a)) - J_l(\sqrt{\la})  (\sqrt{\la} a Y_{l-1}(\sqrt{\la}a) -l  Y_l(\sqrt{\la}a)) \Big ] }\,.
%\end{align*}
Note that  $\la\equiv \lambda_{l,0}(a)$ in the definitions of $P_{\la}$ and $Q_{\la}$.
%In the second line,. 

The value $a=a_l$ is obtained by imposing $\mu_{0,2}(a)= \la_{l,0}(a)$, that is, by solving the system
$$ \Psi(a, \mu)=0\,, \qquad  \Phi(a, \la)=0\,, \qquad \mu = \la\,.$$
In case $l=4$, we have used Mathematica to numerically obtain
\begin{equation} \label{values} a= 0.140989\dots\,, \qquad \la= \mu = 57.5851\dots\,. \end{equation}
We are not giving a computed-assisted proof of this, but it is worth noting that the fact that these values do correspond to the correct roots of $\psi'$ and $\vp$ is clear from the plot given in Figure 2.

%\begin{figure}[h]
%	\centering 
%	\begin{minipage}[c]{120mm}
%		\centering
%		\resizebox{120mm}{80mm}{\includegraphics{picture2.pdf}}
%	\end{minipage}
%	\caption{The plot of $\Psi(r, \mu)$ and $\Phi(r,\la)$ for $a$, $\mu$ and $\la$ as in \eqref{values}. Here $r$ runs from $a$ to $1$. }
%\end{figure}

\begin{figure}[h] 
	\centering 
	\begin{minipage}[c]{105mm}
		\centering
		\resizebox{105mm}{70mm}{\includegraphics{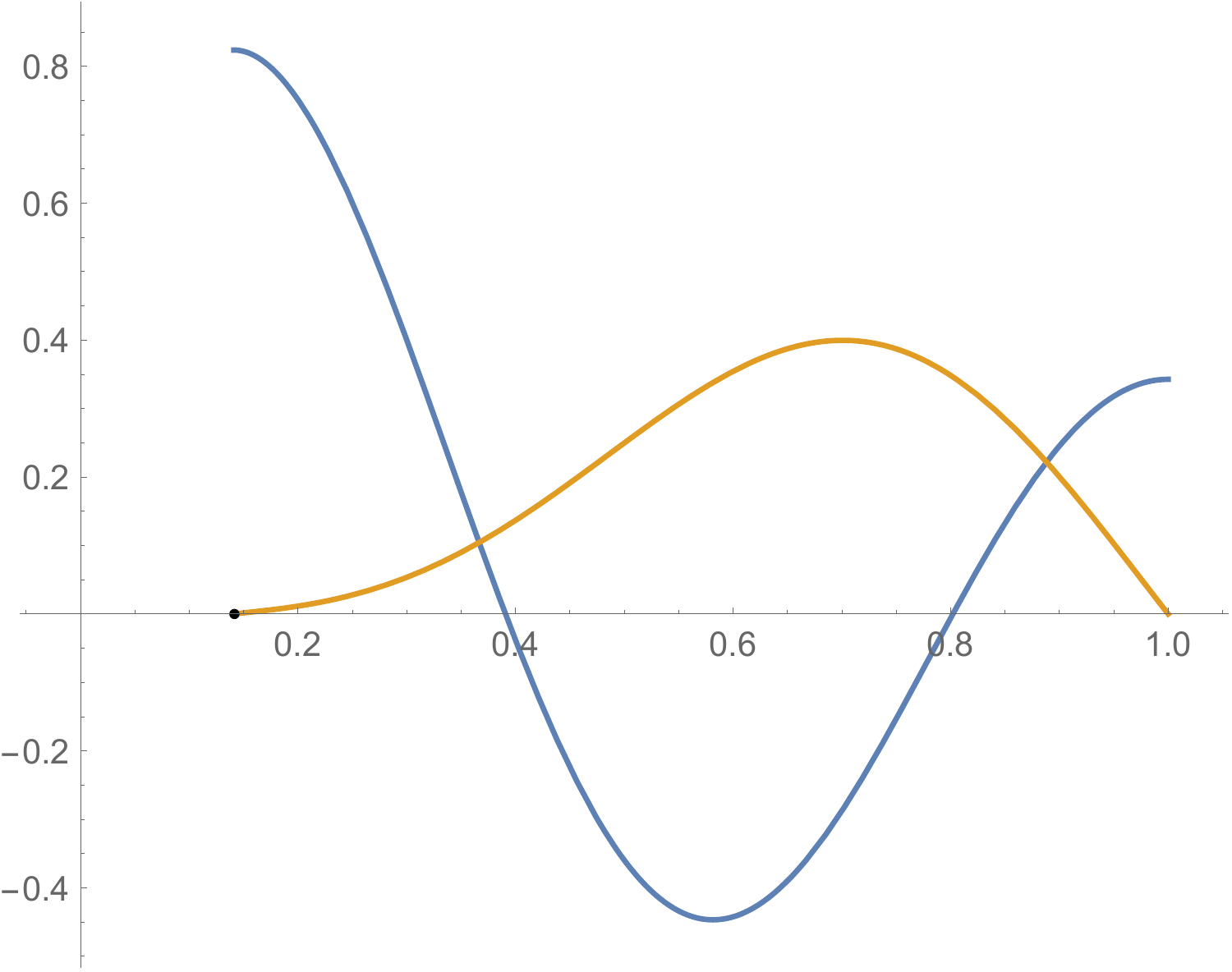}}
	\end{minipage}
	\caption{The plot of the eigenfunctions $\psi(r)$ and $\vp(r)$ for $l=4$ and $a$, $\mu$, $\la$ as in \eqref{values}. Here $r$ runs from the point $a$ to $1$. }
\end{figure}

Plugging \eqref{values} (together with $l=4$) in the expressions of $\mu_{0,2}'(a)$ and $\lambda_{l,0}'(a)$ we obtain:
$$ \mu_{0,2}'(a)= 105.971\dots, \qquad \la_{4,0}'(a)= 0.12067\dots$$
The transversality condition \eqref{T} is clearly satisfied.

We now see (numerically) that in the case \eqref{values}, indeed $\mu=\la$ corresponds to the eigenvalue $\mu(18)$, with $\mu(k):=\mu_k(\Om_{a_4})$. We need to compare the value $\mu$ given in \eqref{values} with the other eigenvalues $\mu_{l,k}(a)$. 

First, note that $\mu\equiv  \mu_{0,2} > \mu_{0, i}$ for $i=0$ or $i=1$. Moreover, by Lemma \ref{extra}, $\mu_{0,2}= \la_{1,1} > \mu_{1,1}$. On the other hand, with our choice of $a$, $\mu_{0,2}= \lambda_{4,0} > \mu_{i,0}$ for $i \leq 4$. Next, we need to compute $\mu_{l,0}(a)$ with $l \geq 5$, and also $\mu_{1,l}(a)$ with $l \geq 2$, where $a$ is given in \eqref{values}. For that, we solve
\begin{equation}
	\psi''+\frac{\psi'}r - l^2 \frac{\psi}{r^2}+ \bar{\mu}\, \psi=0\quad \text{in } (a,1)\,,\qquad \psi'(1)=0\,,
\end{equation}
where $\bar{\mu}$ is real parameter. The function $\psi$ is given by
$$
\begin{aligned}
 \psi(r)= &\ \frac{J_l(\sqrt{\bar{\mu}}r) Y_{l-1}(\sqrt{\bar{\mu}}) - Y_l(\sqrt{\bar{\mu}}r) J_{l-1}(\sqrt{\bar{\mu}})}{Y_{l-1}(\sqrt{\bar{\mu}}) - Y_{l+1}(\sqrt{\bar{\mu}})} \\
 & -\frac{J_l(\sqrt{\bar{\mu}}r) Y_{l+1}(\sqrt{\bar{\mu}}) - Y_l(\sqrt{\bar{\mu}}r) J_{l+1}(\sqrt{\bar{\mu}})}{Y_{l-1}(\sqrt{\bar{\mu}}) - Y_{l+1}(\sqrt{\bar{\mu}})} \,.
\end{aligned}
 $$
The eigenvalues $\mu_{l,i}$ are obtained as the values of $\bar{\mu}$ for which $\psi'(a)=0\,.$ The index $i$ is determined by the number of times that the corresponding eigenfunction changes sign. Using Mathematica, it is easy to find
\begin{align*}
& \mu_{5,0}=41.1601\dots, \qquad  \mu_{6,0}= 56.2689\dots, \qquad  \mu_{7,0}= 73.5792\dots\,,
\end{align*}
and
$$ \mu_{1,2}= 44.0466\dots, \quad  \mu_{1,3}= 64.1201\dots\,.
$$

Therefore, $\mu_{l,0} < \mu_{0,2}$ if and only if $l \leq 6$, and $\mu_{1,l} \leq \mu_{0,2}$ if and only if $l \leq 2$. As the  eigenvalues $\mu_{0,0}$, $\mu_{0,1}$ have multiplicity 1  (because they correspond to radial eigenfunctions) and the rest have multiplicity 2 (since the corresponding eigenspace is spanned by the radial function multiplied by $\sin(l \theta)$ and $\cos(l \theta)$), we conclude that $\mu_{0,2}= \mu_{18}(\Om_{a_4})$. Here we have used that, of course $0=\mu_{0,0}= \mu_0(\Om)$.
\end{proof}

\section{The pulled back problem} \label{A.PullBack}

In this second Appendix we elaborate on the expression of the operator $L_a^{b,B}$ given in Section \ref{S.deform}. First of all, let us recall that in \eqref{E.defPhi} we have set 
$$
\Phi_a^{b,B}(R,\te) = (r(R,\te),\te)\,,
$$
with
\begin{equation} \label{B.r}
r(R,\te) := a+(1-a +B(\te))(2R-1)+2(1-R)b(\te)\,.
\end{equation}
Here, $a \in (0,1)$ and $b ,B \in C^{2,\al}(\TT)$ satisfy 
\[
\|b\|_{L^\infty(\TT)}+\|B\|_{L^\infty(\TT)}<\min\Big\{a,\tfrac{1-a}{2}\Big\}\,.
\]
Then, note that we can as well write
$$
R(r,\te) = \frac{r}{2(1-a+B(\te)-b(\te))} + \frac{1+B(\te)-2(a+b(\te))}{2(1-a+B(\te)-b(\te)}\,,
$$
and so that
\begin{align*}
\pd_r R(r,\te) & = \frac{1}{2(1-a+B(\te)-b(\te))} \,,\\
\pd_\te R(r,\te) & = \frac{B'(\te)(a-r+b(\te))+b'(\te)(r-1-B(\te))}{2(1-a+B(\te)-b(\te))^2}\,, \\
\pd_\te^2 R(r,\te) & = \frac{B''(\te)(a-r+b(\te))+b''(\te)(r-1-B(\te))}{2(1-a+B(\te)-b(\te))^2} \\
& \quad  - \frac{(B'(\te)-b'(\te))(B'(\te)(a-r+b(\te))+b'(\te)(r-1-B(\te)))}{(1-a+B(\te)-b(\te))^3}\,.
\end{align*}
Using that
\begin{align*}
a-r+b(\te) &= (2R-1)(1-a+B(\te)-b(\te)) \,,\\
r-1-B(\te) & = 2(1-R)(1-a+B(\te)-b(\te)) \,, 
\end{align*}
we then infer that
\begin{equation} \label{B.der}
\begin{aligned}
\pd_r R(r(R,\te),\te) & = \frac{1}{2(1-a+B(\te)-b(\te))} \,,\\
\pd_\te R(r(R,\te),\te) & = -\frac{B'(\theta)(2R-1) + 2(1-R)b'(\theta)}{2(1-a+B(\theta)-b(\theta))}\,, \\
\pd_\te^2 R(r(R,\te),\te) & =  \frac{(B'(\theta)-b'(\theta))^2(2R-1)+b'(\theta)(B'(\theta)-b'(\theta))}{(1-a+B(\theta)- b(\theta))^2} \\ 
	& \quad - \frac{B''(\theta)(2R-1) + 2(1-R)b''(\theta)}{2(1-a+B(\theta)-b(\theta))}\,.
\end{aligned}
\end{equation}

Now, for $u \in C^{2,\al}(\overline{\Om}_a^{\,b,B})$, we set 
$$
v := u \circ \Phi_a^{b,B} \in C^{2,\al}(\overline{\Om}_\frac12)\,,
$$
and get that
\begin{align*}
\pd_r^2 & u(r(R,\te),\te) + \frac{1}{r}\pd_r u(r(R,\te),\te) + \frac{1}{r^2} \pd_\te^2 u(r(R,\te),\te) + \mu_{0,2}(a) u(r(R,\te),\te) \\
& =  (\pd_r R)^2 \pd_R^2 v(R,\te) + \frac1{r(R,\te)} (\pd_r R) \pd_R v(R,\te)\\
& \quad + \frac{1}{r(R,\te)^2} \Big( \pd_\te^2 v(R,\te) + (\pd_\te R)^2 \pd_R^2 v(R,\te) \\
& \hspace{2.5cm} + 2 (\pd_\te R) \pd_{R\te} v(R,\te) + (\pd_\te^2 R) \pd_R v(R,\te) \Big) + \mu_{0,2}(a) v(R,\te)\,.
\end{align*}
By substituting \eqref{B.r} and \eqref{B.der} into this expression we get the desired formula for $L_a^{b,B}$. 
 
\bibliographystyle{amsplain}

\end{document}